\documentclass[12pt]{article}

\usepackage[T2A]{fontenc}
\usepackage[utf8]{inputenc}

\usepackage[backend=bibtex8, bibencoding=utf8]{biblatex}
\usepackage{amsmath}
\usepackage{amsthm}
\usepackage{amssymb}

\usepackage{float}
\usepackage{tikz}

\usepackage{url}
\usepackage{hyperref}

\usepackage{caption}
\usepackage{subcaption}

\newtheorem{theorem}{Theorem}[section]
\newtheorem{proposition}{Proposition}[section]
\newtheorem{lemma}{Lemma}[section]
\newtheorem{definition}{Definition}
\newtheorem{remark}{Remark}

\newtheoremstyle{named}{}{}{\itshape}{}{\bfseries}{}{.5em}{\thmnote{#3}}
\theoremstyle{named}
\newtheorem*{namedtheorem}{Theorem}

\title{}
\date{}

\begin{document}
	\begin{center}
		\Large Dimension-free estimates on distances between subsets of volume $\varepsilon$ inside a unit-volume body
	\end{center}

	Abdulamin Ismailov\footnote{E-mail: \href{mailto:nameundefinednamovich@gmail.com}{\nolinkurl{nameundefinednamovich@gmail.com}}} \qquad Alexei Kanel-Belov \qquad Fyodor Ivlev

	\begin{abstract}
		Average distance between two points in a unit-volume body $K \subset \mathbb{R}^n$ tends to infinity as $n \to \infty$.  However, for two small subsets of volume $\varepsilon > 0$ the situation is different. For unit-volume cubes and euclidean balls the largest distance is of order $\sqrt{-\ln \varepsilon}$, for simplexes and hyperoctahedrons -- of order $-\ln \varepsilon$, for $\ell_p$ balls with $p \in [1;2]$ -- of order $(-\ln \varepsilon)^{\frac{1}{p}}$. These estimates are not dependent on the dimensionality $n$. The goal of the paper is to study this phenomenon. Isoperimetric inequalities will play a key role in our approach.
	\end{abstract}

	\tableofcontents

	\newpage

	\section{Introduction.}

	In high dimensions we observe a variety of different phenomena. For example, Vladimir Igorevich Arnold liked to ask his students the following question: <<What percent of the overall mass is occupied by the pulp of the $100$-dimensional watermelon of diameter $1$ meter, if the crust is of width $1$ centimeter?>> The answer is approximately $1 - e^{-1}$. This question in a simple way demonstrates the concentration of measure phenomenon: how most of the mass of a body could lie inside a thin shell. Here is another example, the volume of a euclidean ball of radius $2023$ tends to $0$ as the dimensionality goes to infinity. More generally, we have the isodiametric inequality, which suggests that in high dimensions the diameter of a unit-volume body shall become arbitrarily large.

	The goal of this paper is to achieve a better understanding how things work in high-dimensional spaces by studying the following phenomenon: two points in a unit-volume convex body could be at an arbitrarily large distance from each other; consider, for example, the unit cubes $(0; 1)^n$ -- as $n$ tends to $+\infty$ the diameter equal to $\sqrt{n}$ also tends to infinity, similarly, the average distance would be of order at least $\sqrt{n}$(see Appendix \ref{av_dist}); even the distance between a point and a subset of some fixed volume $\varepsilon < 1$ could be arbitrarily big, but it turns out that the distance between two subsets of some fixed volume $\varepsilon > 0$ in the unit cube is bounded above by some constant dependent on $\varepsilon$ but not on the dimension $n$. What about convex bodies other than the unit cubes?

	Consider a family of unit-volume bounded convex bodies $K_n$. For each $K_n$ it makes sense to consider the supremum of all possible distances between two subsets of some fixed volume $\varepsilon \in (0; \frac{1}{2})$. Denote this value by $d_n(\varepsilon)$.

	By $\Phi$ we mean the function
	\[
		\Phi(a) = \int_{-\infty}^{a} e^{-\pi x^2} dx
	\]
	Function $\Phi^{-1}(\varepsilon)$ is asymptotically equivalent to
	\[
		-\frac{1}{\sqrt{\pi}}\sqrt{-\ln \varepsilon}
	\]
	as $\varepsilon$ tends to $0$(see Appendix \ref{asymp_phi_inv}).

	\begin{namedtheorem}[Theorem \ref{ball_exact}.]
		When $K_n$ are the unit-volume euclidean balls
		\[
			\lim_{n \to \infty} d_n(\varepsilon) = -2\frac{1}{\sqrt{e}} \Phi^{-1}(\varepsilon)
		\]
	\end{namedtheorem}

	\begin{namedtheorem}[Theorems \ref{cube_upper} and \ref{cubes}.]
		When $K_n$ are the unit cubes we have
		\[
			-2\sqrt{\frac{\pi}{6}} \Phi^{-1}(\varepsilon) \leq \liminf_{n \to \infty} d_n(\varepsilon) \leq \limsup_{n \to \infty} d_n(\varepsilon) \leq -2\Phi^{-1}(\varepsilon)
		\]
	\end{namedtheorem}

	\begin{namedtheorem}[Theorems \ref{simplex_upper} and \ref{simplex}.]
		When $K_n$ are the unit-volume simplexes we have
		\[
			-\frac{\sqrt{2}}{e} \ln(2\varepsilon) \leq \liminf_{n \to \infty} d_n(\varepsilon) \leq \limsup_{n \to \infty} d_n(\varepsilon) \leq -c \ln \varepsilon
		\]
		for some universal constant $c > 0$ independent of $n$ and $\varepsilon$.
	\end{namedtheorem}

	\begin{namedtheorem}[Theorems \ref{ell_p_upper} and \ref{l_p_balls}.]
		Fix some $p \in [1;2]$. When $K_n$ are the unit-volume $\ell_p$ balls
		\[
			-2\Psi_p^{-1}(\varepsilon) \leq \liminf_{n \to \infty} d_n(\varepsilon) \leq \limsup_{n \to \infty} d_n(\varepsilon) \leq C_p (-\ln \varepsilon)^{\frac{1}{p}},
		\]
		where $C_p$ is some universal constant determined by $p$, and function $-2\Psi_p^{-1}(\varepsilon)$(see Appendix \ref{v_n_s_n}) is asymptotically equivalent to
		\[
			\frac{1}{e^{\frac{1}{p}} \Gamma\left(1 + \frac{1}{p}\right)} (-\ln \varepsilon)^{\frac{1}{p}}
		\]
		as $\varepsilon \to 0$.
	\end{namedtheorem}

	A version of our problem, in which the euclidean distance is replaced by the Manhattan distance, can be approached by discretization.

	\begin{namedtheorem}[Theorem \ref{discrete_cube}.]
		If by $d_n(\varepsilon)$ we denote the largest Manhattan distance between two bodies of volume $\varepsilon \in (0; \frac{1}{2})$ in the unit cube $[0; 1]^n$, then
		\[
			\lim_{n \to \infty} \frac{d_n(\varepsilon)}{\sqrt{n}} = -2\sqrt{\frac{\pi}{6}}\Phi^{-1}(\varepsilon)
		\]
	\end{namedtheorem}

	We also establish a sort of a general lower bound, showing that in a way euclidean balls are optimal in regard to our problem.

	\begin{namedtheorem}[Theorem \ref{general_sup}.]
		When $K_n$ are unit-volume centrally symmetric bounded convex bodies
		\[
			-2\frac{1}{\sqrt{e}} \Phi^{-1}(\varepsilon) \leq \liminf_{n \to \infty} d_n(\varepsilon)
		\]
	\end{namedtheorem}

	Lower bounds on our problem could be derived simply by considering some hyperplane cuts. But how can we bound above the distance between two subsets $A$ and $B$ in a unit-volume convex body?

	Well, first, we observe that if both $A$ and $B$ are of volume at least $\frac{1}{2}$, then the distance between them is zero(see Lemma \ref{zero_dist}). That is why we assume that both $A$ and $B$ are of some volume $\varepsilon \in (0; \frac{1}{2})$. Next we introduce the concept of a $\delta$-enlargement of a body defined as the set of all points at a distance at most $\delta$ from our body, i.\,e.\,
	\[
		A_{\delta} = \{ x \in X \mid \exists y \in A\colon d(x, y) \leq \delta \}
	\]

	What happens if we replace $A$ with its $\delta$-enlargement for a small enough value of $\delta$? Roughly speaking, a layer of width $\delta$ will be added to our body. The volume of this layer $A_{\delta} \setminus A$ could be approximated as $\delta \cdot S(A)$, where $S(A)$ is the surface area of the body $A$. So the volume of $A$ increases by approximately $\delta \cdot S(A)$, but the distance between bodies $A$ and $B$ will decrease exactly by $\delta$ after we enlarge $A$(or might become zero). To estimate the distance between $A$ and $B$ we will be slowly enlarging them simultaneously until both bodies would be of volume $\frac{1}{2}$ at least, at which point the distance between them is already zero. The double of the amount of time it took both bodies to reach volume $\frac{1}{2}$ would be an upper bound on the distance between them.

	This was just a rough description of how we approach the problem. To make this idea work we are going to need more. We have not said anything about how our bodies may look like, at this point they could be arbitrary subsets of volume $\varepsilon$, which may present a problem, since we plan to rely on concepts such as surface area. In part, these issues might be mitigated by the following observation: after enlarging both $A$ and $B$ by a little $\delta$ distances and volumes would not change much, but smoothness properties might improve. Anyway, throughout this whole text we assume that the bodies we are dealing with are as smooth as needed.

	Now consider the process of a slow enlargement of a body $A$ at its very beginning. Instead of talking about the approximate volume of the layer $A_{\delta} \setminus A$ it would be better to take the right derivative at the point $\delta = 0$. What we will get is called the Minkowski--Steiner formula for the free surface area
	\[
		\mu^+(A) = \lim_{\delta \to 0^+} \frac{\mu(A_{\delta}) - \mu(A)}{\delta}
	\]
	Thus it is vital to our approach to be able to estimate this surface area. But we are only aware of the initial volumes of $A$ and $B$, which leads us to the isoperimetric problem: given the information about the initial volume of a body, find a lower bound on its surface area.

	Euclidean balls have really good symmetry properties. Symmetrization techniques could be applied. Isoperimetric regions inside the euclidean balls have been completely classified(\cite[Theorems 1 and 5]{ros}). This allows us to get tight enough estimates that lead to the proof of Theorem \ref{ball_exact}.

	Unit cubes, however, are not as good as euclidean balls in that regard. That is why instead of dealing with the interior of the unit cube we perform a transfer(Lemma \ref{transfer}) to a different space, where the situation with the isoperimetric problem is better, and by doing so derive the lower bounds on the initial space(\cite[Theorem 7]{ros}).

	At last, to derive lower bounds in the case of simplex new ideas and methods would need to be introduced. Here we repeat the approach from an article by Sasha Sodin \cite{Sodin2008}, where an isoperimetric inequality for $\ell_p$ balls with $p \in [1;2]$ was proven. In particular, in case $p = 1$ we get hyperoctahedrons. Theorem \ref{ell_p_upper} is an immediate consequence of the isoperimetric inequality established in article \cite{Sodin2008}.

	Even though our method does provide asymptotically correct estimates, we should not expect it to lead to exact constants. We bound the growth of $\mu(A_{\delta})$ below by considering the isoperimetric problem for volume $\mu(A_{\delta})$, but that might lead to suboptimal estimates, since as $\delta$ varies $A_{\delta}$ does not have to look like an optimal isoperimetric region.

	\begin{figure}[!h]
		\centering
		\begin{subfigure}[b]{0.28\textwidth}
			\centering
			\begin{tikzpicture}
				\draw (0, 0) -- (0, 2) -- (2, 2) -- (2, 0) -- cycle;
				\draw (0, 1) -- (1, 0);
			\end{tikzpicture}
			\caption{region $A$}
		\end{subfigure}
		\hfill
		\begin{subfigure}[b]{0.28\textwidth}
			\centering
			\begin{tikzpicture}
				\draw (0, 0) -- (0, 2) -- (2, 2) -- (2, 0) -- cycle;

				\draw (0.5, 1.5) -- (1.5, 0.5);
				\draw[domain=0:45] plot ({1 + sqrt(0.5) * cos(\x)}, {sqrt(0.5) * sin(\x)});

				\draw[domain=45:90] plot ({sqrt(0.5) * cos(\x)}, {1 + sqrt(0.5) * sin(\x)});
			\end{tikzpicture}
			\caption{region $A_{\delta}$}
		\end{subfigure}
		\hfill
		\begin{subfigure}[b]{0.42\textwidth}
			\centering
			\begin{tikzpicture}
				\draw (0, 0) -- (0, 2) -- (2, 2) -- (2, 0) -- cycle;
				\draw (0.8, 0) -- (0.8, 2);
			\end{tikzpicture}
			\caption{optimal region of volume $\mu(A_{\delta})$}
		\end{subfigure}
	\end{figure}
	\newpage

	\section{Preliminaries.}

	Assume that we are working in the space $X$ with metric $d$ and probability measure $\mu$, i.\,e.\,$\mu(X) = 1$. In this section we are going to introduce some basic concepts related to our problem.

	\begin{definition}
		The distance between a pair of non-empty subsets $A, B \subseteq X$ is the infimum of distances between points from $A$ and $B$
		\[
			\operatorname{dist}(A, B) = \inf_{x \in A, y \in B} d(x, y)
		\] 
	\end{definition}

	\begin{definition}
		A point $x \in X$ belongs to the $\delta$-enlargement of a subset $A \subseteq X$ if it is at distance at most $\delta$ from some point of $A$
		\[
			A_{\delta} = \{ x \in X \mid \exists y \in A\colon d(x, y) \leq \delta \}
		\]
	\end{definition}

	We want to know how far apart from each other two subsets $A, B \subset X$ of measure $\mu(A) = \mu(B) = \varepsilon > 0$ could be. To that end, note that, if their $\delta$-enlargements intersect, then the distance is bounded above by $2\delta$
	\[
		A_{\delta} \cap B_{\delta} \neq \varnothing \Rightarrow \operatorname{dist}(A, B) \leq 2\delta
	\]
	
	Our problem is concerned with the case of $X$ being an open convex bounded subset of $\mathbb{R}^n$ of unit volume, $d$ being the euclidean metric, and $\mu$ being the Lebesgue measure. In that case the following lemma holds.
	\begin{lemma}
		\label{zero_dist}
		If $A$ and $B$ are two subsets of $X$ with $\mu(A) + \mu(B) \geq 1$, then they are at a distance $0$ from each other
		\[
			\operatorname{dist}(A, B) = 0
		\]
	\end{lemma}
	\begin{proof}
		Assume the contrary. Let the distance between $A$ and $B$ be a positive number
		\[
			r = \operatorname{dist}(A, B) > 0
		\]
		This would mean that our subsets do not intersect, and thus
		\[
			\mu(A) + \mu(B) = \mu(A \cap B) + \mu(A \cup B) = \mu(A \cup B) = 1
		\]

		Pick a pair of points $a \in A$ and $b \in B$ at a distance less than $2r$
		\[
			d(a, b) < 2r
		\]
		By $c$ denote the midpoint of the segment between $a$ and $b$. The distance from $c$ to both subsets $A$ and $B$ is strictly less than $r$, so the point $c$ does not belong to any of our subsets. Now consider a $\delta$-neighborhood of $c$ that lies inside $X$ with $\delta < r - \frac{d(a, b)}{2}$. Clearly, it could not intersect neither $A$ nor $B$, and at the same time it has a non-zero measure, so
		\[
			\mu(X \setminus (A \cup B)) > 0,
		\]
		which leads to contradiction.
	\end{proof}

	To ensure the nonemptiness of the intersection of $A_{\delta}$ and $B_{\delta}$ the following condition would suffice
	\[
		\mu(A_{\delta}) \geq \frac{1}{2} \textrm{ and } \mu(B_{\delta}) \geq \frac{1}{2}
	\]
	Thus we are interested in the growth of $\mu(A_{\varepsilon})$ considered as a function of $\varepsilon$, since that might lead to an upper bound on $\delta$ and consequently on $\operatorname{dist}(A, B)$. The derivative of $\mu(A_{\varepsilon})$ at $\varepsilon = 0$ gives us

	\begin{definition}
		By the surface area of $A \subseteq X$ we mean the following limit
		\[
			\mu^+(A) = \lim_{\varepsilon \to 0^+} \frac{\mu(A_{\varepsilon}) - \mu(A)}{\varepsilon}
		\]
	\end{definition}

	Lower bounds on $\mu^+(A)$ might allow us to get results on the growth of $\mu(A_{\varepsilon})$, but all we know is the measure $\mu(A)$ of our subset $A$. So we want to know the least possible value of $\mu^+(A)$ when $\mu(A)$ is fixed, or at least bound $\mu^+(A)$ below.

	\begin{definition}
		By the isoperimetric profile we mean a function that maps $t$ to the infimum of possible values that $\mu^+(A)$ could take when $\mu(A) = t$.
		\[
			I_{\mu}(t) = \inf_{\mu(A) = t} \mu^+(A)
		\]
	\end{definition}
	\noindent We are no longer interested in the growth of $\mu(A_{\varepsilon})$ after we reach the measure of one half, also our initial $\mu(A)$ is greater than $0$. This means that we are only interested in the values of $I_{\mu}(t)$ when $0 < t < \frac{1}{2}$. That is why throughout this paper by default the domain of the isoperimetric profile is the interval $(0; \frac{1}{2})$.

	A region, which has the minimal surface area amongst all the regions of the same measure, is called an isoperimetric region, and its boundary is called an isoperimetric hypersurface.

	\newpage

	\section{Euclidean balls.}
	\label{euclidean_balls}

	\textbf{Introduction.} The euclidean ball is a perfect candidate for applying the symmetrization techniques. An argument({{\cite[Theorems 1 and 5]{ros}}}) involving them completely classifies the optimal isoperimetric regions of the euclidean ball. Lower bounds on the isoperimetric profile thus could be extracted by considering these optimal regions.

	\medskip

	By $B^n$ we denote the unit $n$-ball. Its volume is
	\[
		\frac{\sqrt{\pi}^n}{\Gamma\left(\frac{n}{2} + 1\right)}
	\]
	So the unit-volume $n$-ball will be of radius
	\[
		\omega_n = \frac{\Gamma\left(\frac{n}{2} + 1\right)^{\frac{1}{n}}}{\sqrt{\pi}} \sim \sqrt{\frac{n}{2\pi e}}
	\]
	By $\mu$ denote the Lebesgue measure on $\omega_n B^n$.

	Combination of theorems $1$ and $5$ from \cite{ros} provides a classification of optimal isoperimetric regions in $\omega_n B^n$.

	\begin{theorem}[{{\cite[Theorems 1 and 5]{ros}}}]
		\label{ball_iso_class}
		Isoperimetric hypersurfaces in a ball are either hyperplanes passing through the origin or spherical caps which are orthogonal to the surface of $\omega_n B^n$.
	\end{theorem}

	We would like to find lower bounds on the isoperimetric profile $I_{\mu}$ of $\omega_n B^n$, by Theorem \ref{ball_iso_class} it would suffice to consider intersections with balls orthogonal to $\omega_n B^n$.

	By $\Psi(x)$ we denote $\Phi(\sqrt{e}x)$. Note that function $\Psi(x)$ has a finite Lipschitz constant $C > 0$, since its derivative is a bounded function. On the interval $(-\infty; 0)$ both $\Psi(x)$ and $\Psi^{\prime}(x)$ are increasing functions.

	\begin{theorem}
		\label{iso_in_ball}
		For every $\varepsilon_0 \in (0; \frac{1}{2})$ and $\tau > 0$ there is a number $N$ such that the isoperimetric inequality
		\[
			I_{\mu}(\Psi(t) + \tau) \geq \Psi^{\prime}(t)
		\]
		would hold for all $n > N$ and $\Psi(t) \in (\varepsilon_0; \frac{1}{2} - \tau)$.
	\end{theorem}
	\begin{proof}
		First, note that $I_{\mu}$ is a non-decreasing function on the interval $(0; \frac{1}{2})$. Indeed, if one would take a ball orthogonal to $\omega_n B^n$ whose intersection with $\omega_n B^n$ is of volume $\varepsilon \in (0; \frac{1}{2})$ and replace it with a ball that has the same radius but whose center is further away from the center of $\omega_n B^n$, one would get a region of $\omega_n B^n$ of smaller volume and smaller surface area.

		Here we are going to prove that for any $D, \delta_1, \delta_2 > 0$ there is a number $N$ such that for every $n > N$ and $0 < d \leq D$ there is an optimal isoperimetric region in $\omega_n B^n$ whose volume $V$ and surface area $S$ satisfy
		\[
			\Psi(-d) + \delta_1 \geq V
		\]
		\[
			S \geq \Psi^{\prime}(-d - \delta_2)
		\]

		Note that our theorem follows from this last claim. Indeed, in the statement of the theorem we require $\Psi(t)$ to be in range $(\varepsilon_0; \frac{1}{2} - \tau)$, which means that $t \geq \Psi^{-1}(\varepsilon_0)$. Set $D = -\Psi^{-1}(\varepsilon_0)$ and pick $\delta_1, \delta_2 > 0$ so that
		\[
			C \delta_2 + \delta_1 \leq \tau
		\]
		\[
			t < \Psi^{-1}\left(\frac{1}{2} - \tau\right) \leq -\delta_2
		\]
		By our claim there will be a number $N$ such that for every $n > N$ and $d \in (0; D]$ there would be an optimal isoperimetric region in $\omega_n B^n$ of volume $V$ not greater than $\Psi(-d) + \delta_1$ and surface area $S$ at least $\Psi^{\prime}(-d - \delta_2)$. Since this region is optimal,
		\[
			I_{\mu}(V) = S
		\]
		And our bounds imply
		\begin{equation}
			\label{VS}
			I_{\mu}(\Psi(-d) + \delta_1) \geq I_{\mu}(V) = S \geq \Psi^{\prime}(-d - \delta_2)
		\end{equation}
		For an arbitrary $t$ satisfying $\Psi(t) \in (\varepsilon_0; \frac{1}{2} - \tau)$ we can set $d = -t - \delta_2 \in (0; D]$, but then
		\[
			\Psi(-d) + \delta_1 = \Psi(t + \delta_2) + \delta_1 \leq \Psi(t) + C\delta_2 + \delta_1 \leq \Psi(t) + \tau
		\]
		We combine this with (\ref{VS}) and get the desired isoperimetric inequality
		\[
			I_{\mu}(\Psi(t) + \tau) \geq I_{\mu}(\Psi(-d) + \delta_1) \geq \Psi^{\prime}(-d - \delta_2) = \Psi^{\prime}(t)
		\]

		Now we need to prove our claim. Fix $D > 0$. We would be considering the intersection of $\omega_n B^n$ and a ball $B^n_r(A)$ orthogonal to it such that the distance from the origin, which we will denote here as $O$, to $B^n_r(A)$ is some number $\frac{d}{2}$ in range $(0; \frac{D}{2}]$.

		Two balls are orthogonal if their centers together with an arbitrary point on the intersection of the corresponding spheres form a right triangle.

		\begin{figure}[H]
			\begin{tikzpicture}
				\draw (0, 0) -- (5, 0) -- (1.8, 2.4) -- cycle;
				\draw (0, 0) -- (5, 0) -- (1.8, -2.4) -- cycle;
				\draw (0, 0) circle (3);
				\draw (5, 0) circle (4);

				\filldraw (0, 0) circle (0.05);
				\node[below] (O) at (0, 0) {$O$};

				\filldraw (5, 0) circle (0.05);
				\node[below] (A) at (5, 0) {$A$};

				\filldraw (1.8, 2.4) circle (0.05);
				\node[above] (X) at (1.8, 2.5) {$X$};

				\filldraw (1.8, -2.4) circle (0.05);
				\node[below] (X) at (1.8, -2.5) {$Y$};

				\filldraw (1, 0) circle (0.05);
				\node[below left] (N) at (1, 0) {$N$};

				\filldraw (1.8, 0) circle (0.05);
				\node[below right] (H) at (1.8, 0) {$H$};

				\filldraw (1.5, 0) circle (0.05);
				\node[below] (M) at (1.5, 0) {$M$};

				\draw (1.8, 2.4) -- (1.8, 0);
				\draw (1.8, -2.4) -- (1.8, 0);
			\end{tikzpicture}
			\centering
			\caption{The right triangle described above}
		\end{figure}
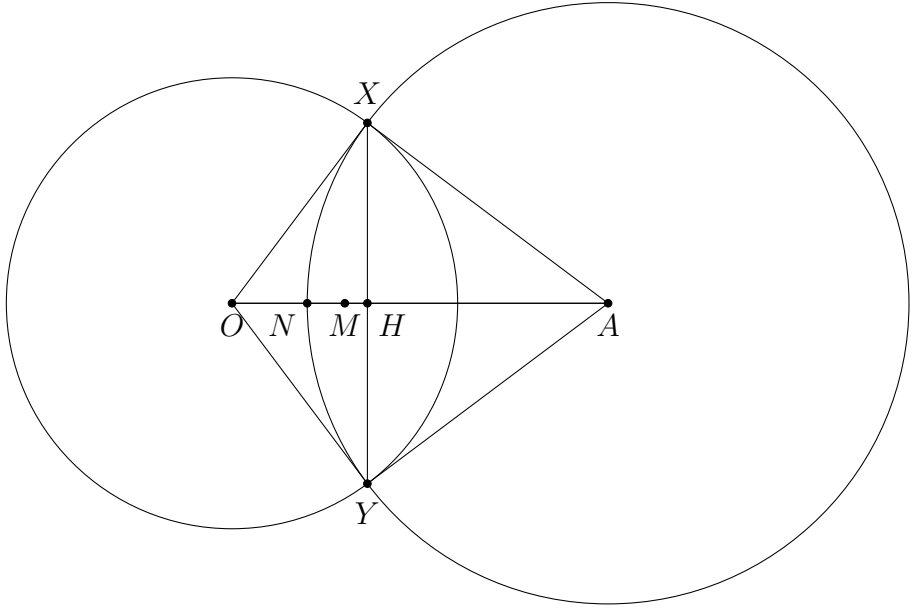
		
		In the figure above ball $\omega_n B^n$ corresponds to the circle with center at $O$ and radius $OX$, ball $B^n_r(A)$ -- to the circle with center at $A$ and radius $AX$. Since we have a right triangle,
		\[
			OA = \sqrt{OX^2 + AX^2} = \sqrt{\omega_n^2 + r^2}
		\]
		We require that $ON = \frac{d}{2}$
		\[
			ON = OA - AN = \sqrt{\omega_n^2 + r^2} - r = \frac{d}{2}
		\]
		\[
			\sqrt{\omega_n^2 + r^2} = r + \frac{d}{2}
		\]
		\[
			\omega_n^2 = rd + \frac{d^2}{4}
		\]
		\[
			r = \frac{\omega_n^2}{d} - \frac{d}{4}
		\]
		Consider the altitude $XH$ of the right triangle $OAX$ and note that
		\begin{multline}
			\label{twice}
			OH = \frac{OX^2}{OA} = \frac{\omega_n^2}{\sqrt{\omega_n^2 + r^2}} = \frac{1}{\sqrt{\frac{1}{\omega_n^2} + \left(\frac{r}{\omega_n^2}\right)^2}} = \frac{1}{\sqrt{\frac{1}{\omega_n^2} + \left( \frac{1}{d} - \frac{d}{4\omega_n^2} \right)^2}} \\ = \frac{1}{\sqrt{\frac{1}{d^2} + \frac{1}{2\omega_n^2} + \frac{d^2}{16\omega_n^4}}} = \frac{1}{\sqrt{\left( \frac{1}{d} + \frac{d}{4\omega_n^2} \right)^2}} = \frac{1}{\frac{1}{d} + \frac{d}{4\omega_n^2}} = \frac{d}{1 + \frac{d^2}{4\omega_n^2}} \leq d
		\end{multline}

		Let $P$ be a hyperplane at a distance $x$ from the origin. By $S_n(x)$ denote the volume of the hyperplane section of $\omega_n B^n$ by $P$. Hyperplane $P$ cuts $\omega_n B^n$ into two parts, at least one which is of volume not greater than $\frac{1}{2}$, denote that volume by $V_n(x)$. Both $V_n(x)$ and $S_n(x)$ are decreasing functions defined on $[0; \infty)$.

		It follows from the proof of theorem $1$ from \cite{ball_section} that the sequence of functions $V_n(x)$ uniformly converges to $\Psi(-x)$ and that the sequence of functions $S_n(x)$ uniformly converges to $\Psi^{\prime}(-x)$(see Appendix \ref{v_n_s_n}).

		On the interval $[0; D]$ positive continuous function $\Psi^{\prime}(-x) - \Psi^{\prime}(-x - \delta_2)$ reaches its minimum value $\varepsilon_1 > 0$. By uniform convergence for all sufficiently large $n$ we shall have
		\[
			S_n(x) \geq \Psi^{\prime}(-x) - \varepsilon_1,
		\]
		which gives us
		\[
			S_n(OH) \geq S_n(d) \geq \Psi^{\prime}(-d) - \varepsilon_1 \geq \Psi^{\prime}(-d - \delta_2)
		\]
		for all $d \in (0; D]$.

		We also know that $S_n(OH) \leq S_n(0)$ and $S_n(0)$ converge towards $\Psi^{\prime}(0)$. Thus $S_n(OH)$ is always bounded above by some constant $S_0$.

		Hyperplane passing through the point $H$ orthogonal to $OA$ divides the intersection of two balls $\omega_n B^n$ and $B^n_r(A)$ into two spherical domes: $\Omega_1$ belonging to $\omega_n B^n$ and $\Omega_2$ belonging to $B^n_r(A)$. The volume of $\omega_n B^n \cap B^n_r(A)$ is equal to the sum of volumes of $\Omega_1$ and $\Omega_2$.

		We would like to bound above the volume of $\Omega_2$. Its base, the hyperplane section passing through $H$, has area not greater than $S_0$. Fix some number $1 > \varepsilon_2 > 0$. On the segment $HN$ pick a point $M$ such that $HM \colon HN = \varepsilon_2$. By $S_1$ denote the area of the hyperplane section of $B^n_r(A)$ passing through $M$ orthogonal to $OA$. Volume of $\Omega_2$ could be bounded above as
		\begin{equation}
			\label{omega_bound}
			S_1 NM + S_0 MH = S_1 NH (1 - \varepsilon_2) + S_0 NH \varepsilon_2
		\end{equation}

		Radius of the $(n-1)$-ball corresponding to $S_1$ equals to
		\begin{multline*}
			\sqrt{AX^2 - AM^2} = \sqrt{AX^2 - (AH + HM)^2} \\ = \sqrt{(AX^2- AH^2) - 2AH \cdot HM - HM^2} \leq \sqrt{XH^2 - 2\varepsilon_2 AH \cdot HN}
		\end{multline*}
		Radius of the $(n-1)$-ball corresponding to $S_0$ is $XH$, so the ratio of the two radii is
		\[
			\sqrt{1 - 2\varepsilon_2 \frac{AH}{XH^2} HN}
		\]
		Since $XH$ is the altitude in the right triangle, $XH^2$ is equal to $OH \cdot HA$, and the ratio could be rewritten as
		\[
			\sqrt{1 - 2\varepsilon_2 \frac{HN}{OH}}
		\]
		Note that $HN = OH - ON$ and that by formula (\ref{twice})
		\begin{multline*}
			\sqrt{1 - 2\varepsilon_2 \frac{HN}{OH}} = \sqrt{1 - 2\varepsilon_2\left(1 - \frac{ON}{\frac{2ON}{1 + \frac{d^2}{4\omega_n^2}}}\right)} = \sqrt{1 - 2\varepsilon_2\left(1 - \frac{1 + \frac{d^2}{4\omega_n^2}}{2}\right)}
		\end{multline*}
		For all sufficiently large $n$
		\[
			\frac{d^2}{4\omega_n^2} \leq \frac{1}{2} \Rightarrow \sqrt{1 - 2\varepsilon_2\left(1 - \frac{1 + \frac{d^2}{4\omega_n^2}}{2}\right)} \leq \sqrt{1 - \frac{\varepsilon_2}{2}}
		\]
		We conclude
		\[
			S_1 \leq S_0 \left( 1 - \frac{\varepsilon_2}{2} \right)^{\frac{n - 1}{2}}
		\]

		Clearly, $NH \leq OH \leq d \leq D$, and by (\ref{omega_bound}) the volume of $\Omega_2$ is not greater than
		\[
			D S_0\left( \left( 1 - \frac{\varepsilon_2}{2} \right)^{\frac{n - 1}{2}}(1 - \varepsilon_2) + \varepsilon_2 \right)
		\]
		Note that we could pick $\varepsilon_2 \in (0; 1)$ so that for all sufficiently large $n$
		\begin{equation}
			\label{ineq_1}
			D S_0\left( \left( 1 - \frac{\varepsilon_2}{2} \right)^{\frac{n - 1}{2}}(1 - \varepsilon_2) + \varepsilon_2 \right) \leq \frac{\delta_1}{2}
		\end{equation}

		The volume of $\Omega_1$ is $V_n(OH)$. And by (\ref{twice})
		\[
			d - OH = d - \frac{d}{1 + \frac{d^2}{4\omega_n^2}} = d \cdot \frac{\frac{d^2}{4\omega_n^2}}{1 + \frac{d^2}{4\omega_n^2}} \leq \frac{D^3}{4\omega_n^2}
		\]
		We noted that the sequence of functions $V_n(x)$ uniformly converges to $\Psi(-x)$. Thus for all sufficiently large $n$
		\[
			\Psi(-OH) + \frac{\delta_1}{4} \geq V_n(OH)
		\]
		Note that for all $x \geq 0$
		\[
			0 < \Psi^{\prime}(-x) \leq \Psi^{\prime}(0),
		\]
		which means
		\begin{equation}
			\label{ineq_2}
			\Psi(-d) + \Psi^{\prime}(0) \frac{D^3}{4\omega_n^2} + \frac{\delta_1}{4} \geq \Psi(-OH) + \frac{\delta_1}{4} \geq V_n(OH)
		\end{equation}
		And for large enough $n$ we have
		\begin{equation}
			\label{ineq_3}
			\Psi^{\prime}(0) \frac{D^3}{4\omega_n^2} \leq \frac{\delta_1}{4}{}
		\end{equation}

		We conclude that the volume $V$ of $\omega_n B^n \cap B^n_r(A)$ is equal to the sum of volumes of $\Omega_1$ and $\Omega_2$, and thus by inequalities (\ref{ineq_1}), (\ref{ineq_2}), (\ref{ineq_3})
		\[
			\Psi(-d) + \delta_1 \geq V
		\]

		By Theorem \ref{ball_iso_class} the intersection between $\omega_n B^n$ and $B^n_r$ is an optimal isoperimetric region in $\omega_n B^n$. Its surface area $S$ is equal to the surface area of the spherical cap corresponding to $\Omega_2$, which can be bounded below by the area of the base of $\Omega_2$, i.\,e.\,$S_n(OH)$. We have thus shown that the volume $V$ and the free surface area $S$ of $\omega_n B^n \cap B^n_r(A)$ satisfy
		\[
			\Psi(-d) + \delta_1 \geq V
		\]
		\[
			S \geq S_n(OH) \geq \Psi^{\prime}(-d - \delta_2)
		\]
		for all sufficiently large $n$ and that the number $d$ here can be chosen here as an arbitrary number from $(0; D]$, which proves our claim.
	\end{proof}

	By $d_n(\varepsilon)$ denote the supremum of all possible distances between two subsets of volume $\varepsilon \in (0; \frac{1}{2})$ inside $\omega_n B^n$.

	Using our isoperimetric inequality we derive
	\begin{theorem}
		\label{euc_ball_dist}
		For every $\varepsilon \in (0; \frac{1}{2})$
		\[
			\limsup_{n \to \infty} d_n(\varepsilon) \leq -2\frac{1}{\sqrt{e}} \Phi^{-1}(\varepsilon),
		\]
		where the function $-2\frac{1}{\sqrt{e}} \Phi^{-1}(\varepsilon)$ is asymptotically equivalent to
		\[
			-2\frac{1}{\sqrt{\pi e}} \sqrt{-\ln \varepsilon}
		\]
		as $\varepsilon \to 0$.
	\end{theorem}
	\begin{proof}
		Pick any $\varepsilon_0 \in (0; \varepsilon)$ and $\tau \in (0; \varepsilon)$. By Theorem \ref{iso_in_ball} there is a number $N$ such that inequality
		\begin{equation}
			\label{here_iso_ball}
			I_{\mu}(\Psi(t) + \tau) \geq \Psi^{\prime}(t)
		\end{equation}
		holds for all $n > N$ and $t$ such that $\Psi(t) \in (\varepsilon_0; \frac{1}{2} - \tau)$.

		Assume that $n > N$. We consider two bodies $A$ and $B$ of volume $\varepsilon$ inside the unit-volume euclidean ball $\omega_n B^n$.

		We are interested in the least values $\delta_A, \delta_B$ such that the $\delta_A$-enlargement of body $A$ in $\omega_n B^n$ will be of volume $\frac{1}{2}$ and $\delta_B$-enlargement of body $B$ will be of volume $\frac{1}{2}$ too. For these enlargements we shall have
		\[
			\operatorname{dist}(A_{\delta_A}, B_{\delta_B}) = 0,
		\]
		from which
		\[
			\operatorname{dist}(A, B) \leq \delta_A + \delta_B
		\]
		follows.

		Isoperimetric inequality (\ref{here_iso_ball}) provides an estimate on the growth of $\delta$-enlargements of our bodies:
		\begin{equation}
			\label{y_delta}
			\partial_+ \mu(A_{\delta}) \geq \Psi^{\prime}(\Psi^{-1}(\mu(A_{\delta}) - \tau))
		\end{equation}
		when $\mu(A_{\delta}) \leq \frac{1}{2}$.

		By $\delta_0$ denote $\Psi^{-1}(\varepsilon - \tau)$. Now consider the function
		\[
			y(\delta) = \Psi(\delta_0 + \delta) + \tau
		\]
		By $\delta_M$ denote the moment when $y$ reaches $\frac{1}{2}$, i.\,e. $y(\delta_M) = \frac{1}{2}$. Assume that $\delta_M < \delta_A$. Functions $\mu(A_{\delta})$ and $y(\delta)$ coincide at $\delta = 0$. Furthermore, because of inequality (\ref{y_delta}), we should have
		\begin{multline*}
			\mu(A_{\delta}) \geq y(\delta) \textrm{ and } \partial_+ \mu(A_{\delta}) \geq \Psi^{\prime}(\Psi^{-1}(\mu(A_{\delta}) - \tau)) \\ \geq \Psi^{\prime}(\Psi^{-1}(y(\delta) - \tau)) = \Psi^{\prime}(\Psi^{-1}(\Psi(\delta_0 + \delta))) = \partial_+ y(\delta)
		\end{multline*}
		for all $\delta \in [0; \delta_M]$. But then we have a contradiction
		\[
			\frac{1}{2} = \mu(A_{\delta_A}) > \mu(A_{\delta_M}) \geq y(\delta_M) = \frac{1}{2}
		\]

		Thus $\delta_A$ and similarly $\delta_B$ are bounded above by $\delta_M$, which implies
		\[
			\operatorname{dist}(A, B) \leq \delta_A + \delta_B \leq 2\delta_M
		\]
		The value $\delta_M$ satisfies
		\[
			\frac{1}{2} = y(\delta_M) = \Psi(\delta_0 + \delta_M) + \tau
		\]
		\[
			\frac{1}{2} - \tau = \Psi(\delta_0 + \delta_M)
		\]
		\[
			\delta_M = \Psi^{-1}\left(\frac{1}{2} - \tau\right) - \Psi^{-1}\left(\varepsilon - \tau\right)
		\]
		As $\tau$ tends to $0$
		\[
			\Psi^{-1}\left(\frac{1}{2} - \tau\right) - \Psi^{-1}\left(\varepsilon - \tau\right) \to -\Psi^{-1}(\varepsilon) = -\frac{1}{\sqrt{e}} \Phi^{-1}(\varepsilon)
		\]
		And since we can choose $\tau$ to be an arbitrary number in $(0; \varepsilon)$
		\[
			\limsup_{n \to \infty} d_n(\varepsilon) \leq -2\frac{1}{\sqrt{e}} \Phi^{-1}(\varepsilon)
		\]
	\end{proof}

	\medskip

	\textbf{Conclusions.} Symmetrization techniques lead to the solution of the isoperimetric problem in other different cases: the classical isoperimetric problem in $\mathbb{R}^n$; isoperimetric inequality on the sphere(\cite[Appendix]{FLM}, \cite[Theorem 2.2.1]{gine_nickl}), from which the gaussian isoperimetric inequality could be derived(\cite[Theorem 2.2.3]{gine_nickl}, \cite[Theorem 20]{ros}).

	\newpage

	\section{Unit cubes.}

	\textbf{Introduction.} Unlike the euclidean ball the cube does not have <<many>> symmetries. To derive lower bounds on the isoperimetric profile we are going to perform a <<transfer>> to a different space. Descriptions of this idea could be found in \cite[Theorem 7]{ros}, \cite[Proposition 2.8]{L01}. In this section we will follow the approach presented in \cite{ros}.

	\medskip

	Consider an $n$-dimensional unit cube $(0; 1)^n$. We can think of it as of a space with Lebesgue measure $\mu$ and Euclidean metric. Now we would like to be able to show estimates on the isoperimetric profile of $\mu$. However, it is quite unclear how to deal with the corresponding space. For example, the cube only has a finite number of symmetries, so symmetrization methods would not get us far. That is why it makes sense to consider a way to transfer to a different, <<better>> space -- an idea that plays a key role in our approach.

	\begin{proposition}[{{\cite[Proposition 1]{ros}}}]
		\label{transfer}
		Assume that for a pair of spaces $M$ and $M^{\prime}$ with measures $\upsilon$ and $\upsilon^{\prime}$, respectively, we have a map $\phi: M \to M^{\prime}$ which transforms measure $\upsilon$ into $\upsilon^{\prime}$, i.\,e. $\mu^{\prime}(A) = \mu(\phi^{-1}(A))$, and that is also $c$-Lipschitz for some $c > 0$, i.\,e. a pair of points in $M$ at a distance $d$ has images at distance at most $c \cdot d$. The following inequality holds
		\[
			I_{\upsilon} \leq c \cdot I_{\upsilon^{\prime}}
		\]
	\end{proposition}
	\begin{proof}
		Consider a closed $R^{\prime} \subseteq M^{\prime}$ and its preimage $R = \phi^{-1}(R^{\prime})$. Since $\phi$ transforms $\upsilon$ into $\upsilon^{\prime}$, we shall have $\upsilon^{\prime}(R^{\prime}) = \upsilon(R)$. The fact that $\phi$ is $c$-Lipschitz gives us $\phi(R_{\varepsilon}) \subseteq R^{\prime}_{c \varepsilon}$, from which it follows that
		\[
			\upsilon^{\prime}(R^{\prime}_{c \varepsilon}) = \upsilon(\phi^{-1}(R^{\prime}_{c \varepsilon})) \geq \upsilon(\phi^{-1}(\phi(R_{\varepsilon}))) \geq \upsilon(R_{\varepsilon})
		\]
		By combining this inequality with $\upsilon^{\prime}(R^{\prime}) = \upsilon(R)$ we get
		\[
			\frac{\upsilon^{\prime}(R^{\prime}_{c\varepsilon}) - \upsilon^{\prime}(R^{\prime})}{c \varepsilon} \geq \frac{\upsilon(R_{\varepsilon}) - \upsilon(R)}{c \varepsilon}
		\]
		And by taking limit $\varepsilon \to 0$ we reach conclusion
		\[
			(\upsilon^{\prime})^{+}(R^{\prime}) \geq \frac{1}{c}\upsilon^+(R)
		\]
		So for every closed $R^{\prime} \subseteq M^{\prime}$ we can find $R \subseteq M$ that has the same measure and whose surface area is at most $c$ times the surface area of $R^{\prime}$. Thus we shall have
		\[
			I_{\upsilon}(t) \leq c I_{\upsilon^{\prime}}(t)
		\]
	\end{proof}

	We are going to apply the above lemma to get lower bounds on $I_{\mu}$. The role of $M^{\prime}$ will play our cube with the Lebesgue measure $\mu$ on it. The role of $M$ will play the space $\mathbb{R}^n$ with Gaussian measure $\gamma_n$ defined by its density at a point $x = (x_1, \ldots, x_n)$ as
	\[
		\frac{d\gamma_n}{d x} = e^{-\pi(x_1^2 + \ldots + x_n^2)}
	\]

	In the one-dimensional case the map $\Phi$ defined by
	\[
		\Phi(a) = \int_{-\infty}^{a} e^{-\pi x^2} dx
	\]
	transforms $(-\infty; +\infty)$ into $(0; 1)$. Also $\Phi$ turns Gaussian measure $\gamma_1$ on $(-\infty; +\infty)$ into the Lebesgue measure on $(0; 1)$. Indeed, the Gaussian measure of the segment $[a; b]$ is equal to the integral
	\[
		\int_{a}^{b} e^{-\pi x^2} dx
	\]
	of its density, which in turn is equal to $\Phi(b) - \Phi(a)$, but the image of $[a; b]$ under our map $\Phi$ is $[\Phi(a); \Phi(b)]$, whose Lebesgue measure is equal to $\Phi(b) - \Phi(a)$.

	Now the role of $\phi$ in the above lemma will be played by
	\[
		\phi(x_1, \ldots, x_n) = (\Phi(x_1), \ldots, \Phi(x_n)),
	\]
	i.\,e. we are applying $\Phi$ coordinatewise. It indeed transforms $(-\infty; +\infty)^n = \mathbb{R}^n$ into $(0; 1)^n$. For every box $[a_1; b_1] \times \ldots \times [a_n; b_n]$ we could note that
	\begin{multline*}
		\gamma_n([a_1; b_1] \times \ldots \times [a_n; b_n]) = \int_{a_1}^{b_1} \ldots \int_{a_n}^{b_n} e^{-\pi x_1^2} \cdot \ldots \cdot e^{-\pi x_n^2} dx_n \ldots dx_1 \\ = \left(\int_{a_1}^{b_1} e^{-\pi x_1^2}dx_1\right) \ldots \left(\int_{a_n}^{b_n} e^{-\pi x_n^2}dx_n\right) \\ = (\Phi(b_1) - \Phi(a_1)) \ldots (\Phi(b_n) - \Phi(a_n)) = \mu(\phi([a_1; b_1] \times \ldots \times [a_n; b_n])),
	\end{multline*}
	so $\phi$ turns measure $\gamma_n$ into $\mu$. And finally, map $\phi$ is $1$-Lipschitz since $|\Phi^{\prime}(x)| = |e^{-\pi x^2}| \leq 1$ and $\phi$ applies $\Phi$ coordinatewise.

	As we see the requirements of Proposition \ref{transfer} are met. So the isoperimetric profile $I_{\mu}$ of our unit cube could be bounded below by $I_{\gamma_n}$. But what do we know about $I_{\gamma_n}$? Well, there are tight inequalities on the isoperimetric profile of Gaussian measures, but here we would only need the following theorem.

	\begin{theorem}[{{\cite[Lemma 2.2.2]{gine_nickl}}}]
		\label{gaussian_iso_ineq}
		Let $\gamma^n$ be the standard Gaussian measure defined by its density at a point $x = (x_1, \ldots, x_n) \in \mathbb{R}^n$ as
		\[
			\frac{d\gamma^n}{dx} = \frac{1}{\sqrt{2\pi}^n} e^{-\frac{1}{2}(x_1^2 + \ldots + x_n^2)}
		\]
		Amongst all subsets $A \subset \mathbb{R}^n$ with fixed measure $\gamma_n(A) \in (0; \frac{1}{2})$ the minimum surface area is attained at half-spaces.
	\end{theorem}

	In general, a Gaussian measure $\gamma_{\mu, \sigma^2}^n$ is a measure defined by its density at a point $x \in \mathbb{R}^n$ as
	\[
		\frac{d \gamma_{\mu, \sigma^2}^n}{dx} = \frac{1}{\sqrt{2 \pi \sigma^2}^n}e^{-\frac{1}{2\sigma^2}\Vert x - \mu \Vert^2}
	\]
	But Gaussian measures are equivalent to each other under translation and scaling. For example, if we shrink the standard Gaussian measure $\gamma^n$ by a factor of $\sqrt{2\pi}$, the density $\rho^{\prime}$ of the resulting measure is related to the density $\rho$ of $\gamma^n$ as
	\[
		\rho^{\prime}(x) = \sqrt{2 \pi}^n \rho(\sqrt{2\pi}x) = \sqrt{2\pi}^n \frac{1}{\sqrt{2\pi}^n} e^{-\frac{1}{2}\Vert \sqrt{2\pi} x \Vert^2} = \frac{d\gamma_n}{dx}
	\]
	And we could also note that for $A \subset \mathbb{R}^n$
	\[
		\gamma^n(A) = \gamma_n\left(\frac{1}{\sqrt{2\pi}}A\right) \quad (\gamma^n)^+(A) = \frac{1}{\sqrt{2\pi}}\gamma_n^+\left(\frac{1}{\sqrt{2\pi}}A\right)
	\]
	So shrinking everything by a factor of $\sqrt{2\pi}$ in Theorem \ref{gaussian_iso_ineq} would not change the fact that half-spaces are optimal solutions to the isoperimetric problem. And that is why to figure out lower bounds on $I_{\gamma_n}$ we would only need to consider half-spaces.

	The density of $\gamma_n$ at a point $x$ only depends on $\Vert x \Vert$, so our measure is rotation-invariant, which means that we could only consider half-spaces $H_a$ defined by $x_n \leq a$ for some $a$. First, we could note that
	\begin{multline*}
		\gamma_n(H_a) = \gamma_n(\underbrace{(-\infty; +\infty) \times \ldots \times (-\infty; +\infty)}_{n - 1 \textrm{ times}} \times (-\infty; a]) \\ = \underbrace{\gamma_1((-\infty; +\infty)) \times \ldots \times \gamma_1((-\infty; +\infty))}_{n - 1 \textrm{ times}} \times \gamma_1((-\infty; a]) \\ = \gamma_1((-\infty; a]) = \Phi(a)
	\end{multline*}
	And, because of this last observation,
	\begin{multline*}
		\gamma_n^+(H_a) = \lim_{\varepsilon \to 0} \frac{\gamma_n((H_a)_{\varepsilon}) - \gamma_n(H_a)}{\varepsilon} = \lim_{\varepsilon \to 0} \frac{\gamma_n(H_{a + \varepsilon}) - \gamma_n(H_a)}{\varepsilon} \\ = \lim_{\varepsilon \to 0} \frac{\gamma_1((-\infty; a + \varepsilon]) - \gamma_1((-\infty; a])}{\varepsilon} \\ = \lim_{\varepsilon \to 0} \frac{\gamma_1((-\infty; a]_{\varepsilon}) - \gamma_1((-\infty; a))}{\varepsilon} = \gamma_1^+((-\infty; a])
	\end{multline*}

	Equalities
	\[
		\gamma_n(H_a) = \gamma_1((-\infty; a]) \quad \gamma_n^+(H_a) = \gamma_1^+((-\infty; a])
	\]
	imply
	\[
		I_{\gamma_n} = I_{\gamma_1}
	\]
	And now we only have to estimate $I_{\gamma_1}$. To do that we need to consider intervals $(-\infty; a]$ for $a < 0$. The measure of such an interval is $\Phi(a)$ and the surface area is $\Phi^{\prime}(a) = e^{-\pi a^2}$, from which we get
	\[
		I_{\gamma_1}(\Phi(a)) = e^{-\pi a^2}
	\]

	By combining our observations we conclude
	\begin{theorem}[{{\cite[Theorem 7]{ros}}}]
		\label{mu_iso}
		For the Lebesgue measure $\mu$ on the unit cube $(0; 1)^n$ isoperimetric inequality
		\[
			I_{\mu}(t) \geq e^{-\pi \Phi^{-1}(t)^2}
		\]
		holds for all $t \in (0; \frac{1}{2})$.
	\end{theorem}
	\begin{proof}
		Note that for all $a < 0$
		\[
			I_{\mu}(\Phi(a)) \geq I_{\gamma_n}(\Phi(a)) = I_{\gamma_1}(\Phi(a)) = e^{-\pi a^2}
		\]
	\end{proof}

	Using this isoperimetric inequality we derive
	\begin{theorem}
		\label{cube_upper}
		Inside a unit cube $(0; 1)^n$ two bodies $A$ and $B$ of volume $\varepsilon \in (0; \frac{1}{2})$ are at a distance at most
		\[
			-2\Phi^{-1}(\varepsilon)
		\]
		Function $-2\Phi^{-1}(\varepsilon)$ is asymptotically equivalent to
		\[
			\frac{2}{\sqrt{\pi}} \sqrt{-\ln \varepsilon}
		\]
		as $\varepsilon \to 0$.
	\end{theorem}
	\begin{proof}
		We are interested in the least values $\delta_A, \delta_B$ such that the $\delta_A$-enlargement of body $A$ in the unit cube $(0; 1)^n$ will be of volume $\frac{1}{2}$ and the $\delta_B$-enlargement of body $B$ will be of volume $\frac{1}{2}$ too. For these enlargements we shall have
		\[
			\operatorname{dist}(A_{\delta_A}, B_{\delta_B}) = 0,
		\]
		from which
		\[
			\operatorname{dist}(A, B) \leq \delta_A + \delta_B
		\]
		follows.

		Isoperimetric inequality from Theorem \ref{mu_iso} provides an estimate on the growth of $\delta$-enlargements of our bodies:
		\begin{equation}
			\label{mu_diff_ineq}
			\partial_+ \mu(A_{\delta}) \geq e^{-\pi \Phi^{-1}(\mu(A_{\delta}))^2}
		\end{equation}
		when $\mu(A_{\delta}) < \frac{1}{2}$.

		By $\delta_M$ denote $-\Phi^{-1}(\varepsilon)$. Now consider the function
		\[
			y(\delta) = \Phi(-\delta_M + \delta)
		\]
		Assume that $\delta_M < \delta_A$. Functions $\mu(A_{\delta})$ and $y(\delta)$ coincide at $\delta = 0$. Furthermore, because of inequality (\ref{mu_diff_ineq}), we should have
		\begin{multline*}
			\mu(A_{\delta}) \geq y(\delta) \textrm{ and } \partial_+ \mu(A_{\delta}) \geq e^{-\pi \Phi^{-1}(\mu(A_{\delta}))^2} \geq e^{-\pi \Phi^{-1}(y(\delta))^2} \\ = e^{-\pi (-\delta_M + \delta)^2} = \partial_+ y(\delta)
		\end{multline*}
		for all $\delta \leq \delta_M$. But then we have a contradiction
		\[
			\frac{1}{2} = \mu(A_{\delta_A}) > \mu(A_{\delta_M}) \geq y(\delta_M) = \Phi(0) = \frac{1}{2}
		\]

		Thus $\delta_A$ and similarly $\delta_B$ are bounded above $\delta_M$, which implies
		\[
			\operatorname{dist}(A, B) \leq \delta_A + \delta_B \leq 2\delta_M = -2\Phi^{-1}(\varepsilon)
		\]
	\end{proof}

	\medskip

	\textbf{Conclusions.} 
	The transfer from one space to another might have lead to the loss of accuracy to some extent, so there is not much of what we could say about how precise our estimates are. The problem of finding optimal hypersurfaces in the $n$-cube also seems to be quite complicated. More details about the isoperimetric inequalities in a cube and their applications the reader may find in section $1.5$ of \cite{ros}.

	In a similar way we could derive lower bounds on the isoperimetric profile of the euclidean ball, since the transition to the space $\mathbb{R}^n$ with gaussian measure is possible(see \cite[Proposition 2.9]{L01}).

	\newpage

	\section{Simplexes and $\ell_p$-balls.}

	\textbf{Introduction.} In this section we are going to describe the approach used in \cite{Sodin2008} to prove an isoperimetric inequality for $\ell_p$ balls($p \in [1;2]$) by presenting a very similar proof of an isoperimetric inequality for simplexes. Much like in the case of the unit cube we will be performing a transfer to a different space(see Lemmas \ref{expon_into_delta}, \ref{profile_of_nu}). But to address the problems with Lipschitz continuity(see Lemma \ref{lipschitz_constant}) a number of new ideas and methods needs to be introduced.

	\medskip

	\subsection{Simplexes.}

	By $\mathbb{R}_+$ we denote the interval $(0; +\infty)$. Consider a regular simplex $\Delta_n$ defined as
	\[
		\Delta_n = \{ (x_1, \ldots, x_n) \in \mathbb{R}_+^n \mid x_1 + \ldots + x_n = 1 \}
	\]
	By $\mu$ we will denote the normalized Lebesgue measure on $\Delta_n$. Note that $\mu(\Delta_n) = 1$. Simple calculations show that the area of $\Delta_n$ is equal to $\frac{n\sqrt{n}}{n!}$. So if we set
	\[
		\omega_n = \left(\frac{n!}{n\sqrt{n}}\right)^{\frac{1}{n - 1}} \sim \frac{n}{e},
	\]
	the area of $\omega_n \Delta_n$ will be equal to $1$. By $\lambda$ denote the Lebesgue measure on $\omega_n \Delta_n$. One could note
	\begin{multline}
		\label{mu_lambda_det}
		\mu^+\left(\frac{1}{\omega_n} A\right) = \lim_{\varepsilon \to 0} \frac{\mu\left(\left(\frac{1}{\omega_n} A\right)_{\varepsilon}\right) - \mu\left(\frac{1}{\omega_n} A\right)}{\varepsilon} \\ = \lim_{\varepsilon \to 0} \frac{\lambda(A_{\omega_n \varepsilon}) - \lambda(A)}{\varepsilon} = \omega_n \lambda^+(A)
	\end{multline}
	for $A \subseteq \Delta_n$. 

	In this section we will be using a slightly different notion of an isoperimetric profile.

	\begin{definition}
		By the isoperimetric function we mean a function that maps $t \in (0; \frac{1}{2})$ to the infimum of possible values that $\mu^+(A)$ could take when $t \leq \mu(A) < \frac{1}{2}$
		\[
			\mathcal{I}_{\mu}(t) = \inf_{t \leq \mu(A) < \frac{1}{2}} \mu^+(A)
		\]
	\end{definition}

	Our observation (\ref{mu_lambda_det}) implies
	\begin{equation}
		\label{mu_lambda}
		\mathcal{I}_{\mu} = \omega_n \mathcal{I}_{\lambda},
	\end{equation}
	i.\,e. the isoperimetric functions are proportional.

	To solve our problem we would need estimates on $\mathcal{I}_{\lambda}$, but for the sake of simplicity we would be working with $\mathcal{I}_{\mu}$ instead. Yet again it is quite unclear how to deal with $\Delta_n$ as a space, so we would like to be able to transfer to a <<better>> space.

	The map $T\colon \mathbb{R}_+^n \to \Delta_n$ defined as
	\[
		T(x_1, \ldots, x_n) = \left(\frac{x_1}{x_1 + \ldots + x_n}, \ldots, \frac{x_n}{x_1 + \ldots + x_n} \right)
	\]
	transforms the measure $\nu_n$ on $\mathbb{R}_+^n$ defined by its density at a point $(x_1, \ldots, x_n) \in \mathbb{R}_+^n$ as
	\[
		\frac{d\nu_n}{dx} = e^{-x_1-\ldots-x_n}
	\]
	into the normalized Lebesgue measure $\mu$ on $\Delta_n$ as a corollary of the following lemma.

	\begin{lemma}[{{\cite[Lemma 2.1]{SZ}}}]
		\label{expon_into_delta}
		Let $X_1, \ldots, X_n$ be independent random variables each with density function $\frac{1}{2}e^{-|t|}$ and put $S = \sum_{i} | X_i |$. Then $(\frac{X_1}{S}, \ldots, \frac{X_n}{S})$ induces the normalized Lebesgue measure on the surface of $\ell_1^{n}$ ball. Moreover, $(\frac{X_1}{S}, \ldots, \frac{X_n}{S})$ is independent of $S$.
	\end{lemma}

	\noindent Indeed, because of the symmetry amongst the orthants of $\mathbb{R}^n$,\footnote{the orthants of $\mathbb{R}^n$ are multidimensional analogues of the quadrants of $\mathbb{R}^2$} we can restrict our attention to the positive orthant $\mathbb{R}_+^n$ in Lemma \ref{expon_into_delta} and reach the desired conclusion.

	But what do we know about the isoperimetric profile of $I_{\nu_n}$? The following lemma completely determines this isoperimetric profile in the one-dimensional case.

	\begin{lemma}[{{\cite[Remark 1]{expm}}}]
		\label{profile_of_nu}
		By $\nu$ denote $\nu_1$, then
		\[
			I_{\nu}(t) = \min(t, 1 - t),
		\]
		where the domain of $I_{\nu}$ is the whole interval $(0; 1)$.
	\end{lemma}

	For a measure $\upsilon$ we could consider its isoperimetric constant -- the largest value $Is(\upsilon)$ for which the following holds for all subsets $A$ with $\upsilon(A) \in (0; 1)$
	\[
		\upsilon^+(A) \geq Is(\upsilon)\min(\upsilon(A), 1 - \upsilon(A))
	\]
	And by Lemma \ref{profile_of_nu} we have $Is(\nu) = 1$. Now we could note that\footnote{by this we mean the product measure $\underbrace{\nu \times \ldots \times \nu}_{n \textrm{ times }}$} $\nu_n = \nu^n$ since the density of $\nu_n$ at a point $(x_1, \ldots, x_n) \in \mathbb{R}_+^n$ could be written as a product
	\[
		e^{-x_1} \ldots e^{-x_n}
	\]

	And thus the following theorem
	\begin{theorem}[{{\cite[Theorem 1.1]{prod}}}]
		For triple $(X, d, \psi)$ -- space, metric, measure,
		\[
			Is(\psi^n) \geq \frac{1}{2 \sqrt{6}} Is(\psi)
		\]
	\end{theorem}
	\noindent gives us
	\[
		I_{\nu^n}(t) \geq \frac{1}{2\sqrt{6}} \min(t, 1 - t)
	\]
	or\footnote{recall that, generally, when we talk about isoperimetric profiles we are only interested in the values of $t \in (0; \frac{1}{2})$, i.\,e. the domain of $I_{\nu_n}$ is $(0; \frac{1}{2})$}
	\begin{equation}
		\label{nu_isoperimetric}
		I_{\nu_n}(t) \geq \frac{1}{2\sqrt{6}}t
	\end{equation}

	So we have a map $T\colon \mathbb{R}_+^n \to \Delta_n$ that transforms $\nu_n$ into $\mu$ and a lower bound on $I_{\nu_n}(t)$. But to use Proposition \ref{transfer} we would also need our mapping $T$ to be Lipschitz continuous.

	In the neighborhood of a point $x \in \mathbb{R}_+^n$ the behavior of our map $T$ could be described by a linear operator defined by the matrix, whose entries are
	\[
		\frac{\partial T_j(x)}{\partial x_i},
	\]
	where $T_j(x)$ is the $j$-th coordinate of $T(x)$.
	We are interested in the norm of this linear operator. The next lemma gives an upper bound

	\begin{lemma}[corresponds to Lemma 1 from \cite{Sodin2008}]
		\label{lipschitz_constant}
		\[
			\left\Vert \frac{\partial T_j(x)}{\partial x_i} \right\Vert_2 \leq \frac{1}{\Vert x \Vert_1}(1 + \sqrt{n}\Vert T(x) \Vert_2)
		\]
	\end{lemma}
	\begin{proof}
		First, we will calculate the entries of our matrix, i.\,e.\,the partial derivatives
		\begin{multline*}
			\frac{\partial}{\partial x_i} T_j(x) = \frac{\partial}{\partial x_i} \frac{x_j}{x_1 + \ldots + x_n} \\ = \frac{(x_1 + \ldots + x_n)\frac{\partial}{\partial x_i} x_j - x_j\frac{\partial}{\partial x_i}(x_1 + \ldots + x_n)}{(x_1 + \ldots + x_n)^2} \\ = \frac{1}{x_1 + \ldots + x_n}\left(\delta_{ij} - \frac{x_j}{x_1 + \ldots + x_n}\right)
		\end{multline*}
		If $\Delta y$ is the image of $\Delta x$, then
		\begin{equation}
			\label{delta_y_j}
			\Delta y_j = \sum_{i} \frac{\partial T_j}{\partial x_i} \Delta x_i = \frac{1}{\Vert x \Vert_1}\left(\Delta x_j - \frac{x_j}{\Vert x \Vert_1}\sum_{i} \Delta x_i\right)
		\end{equation}
		The length of the vector, whose coordinates are $\frac{x_j}{\Vert x \Vert_1} \sum_{i} \Delta x_i$, could be estimated as
		\[
			\Vert T(x) \Vert_2 \left| \sum_{i} \Delta x_i \right| \leq \sqrt{n} \Vert T(x) \Vert_2 \Vert \Delta x \Vert_2
		\]
		since coordinates $\frac{x_j}{\Vert x \Vert_1}$ define $T(x)$ and, clearly,
		\[
			\sum_{i} \left|\Delta x_i\right| \leq \sqrt{n} \left(\sum_{i} (\Delta x_i)^2\right)^{\frac{1}{2}}
		\]
		And now by triangle inequality from (\ref{delta_y_j}) we get
		\begin{multline*}
			\Vert \Delta y \Vert_2 \leq \frac{1}{\Vert x \Vert_1}\left( \Vert \Delta x \Vert_2 + \sqrt{n} \Vert T(x) \Vert_2 \Vert \Delta x \Vert_2 \right) \\ = \frac{1}{\Vert x \Vert_1}\left( 1 + \sqrt{n} \Vert T(x) \Vert_2 \right) \Vert \Delta x \Vert_2,
		\end{multline*}
		from which the statement of the lemma follows.
	\end{proof}

	Indeed, we cannot apply Proposition \ref{transfer} in our case since there are places where the image of the mapping $T$ varies wildly. For example, the simplex in $\mathbb{R}_+^n$ defined by $x_1 + \ldots + x_n = \delta$, where $\delta$ is a very small number.

	However, the upper bound from Lemma \ref{lipschitz_constant}
	\[
		\left\Vert \frac{\partial T_j(x)}{\partial x_i} \right\Vert_2 \leq \frac{1}{\Vert x \Vert_1}(1 + \sqrt{n}\Vert T(x) \Vert_2)
	\]
	tells us that the only parts of $\mathbb{R}_+^n$, where the variance of $T$ could be high, are the regions where $\Vert x \Vert_1$ is too small, which corresponds to a little corner of the orthant $\mathbb{R}_+^n$, or where $\Vert T(x) \Vert_2$ is too large, and, since the distance from the origin to the center of $\Delta_n$ -- $\frac{1}{\sqrt{n}}$ is much smaller than the distance from the origin to the vertices of $\Delta_n$ -- $1$, the region where $\Vert T(x) \Vert_2$ is too large corresponds to the little corners of $\Delta_n$ near the vertices.\footnote{here we are talking about the image of $T$}

	Our last observation suggests that the parts of our space, where $T$ does not behave the way we want it to, i.\,e.\,high variance, might be negligible. And now we are going to present an argument that will allow us to <<get rid>> of these regions, perform the transfer to the space $\mathbb{R}_+^n$ with measure $\nu_n$ and then apply the lower bound for $I_{\nu_n}$.

	But, first of all, since we want to estimate $\mathcal{I}_{\mu}(t)$ on the whole half-interval $(0; \frac{1}{2}]$, including the small values of $t$, and since we are intending to remove some negligible parts of our space in the main argument, we need to employ a different approach for those values of $t$ that might, perhaps, be lesser than the measure of the regions that we are getting rid of. In other words, an estimate on the surface area of the <<small>> subsets of $\Delta_n$.

	This could be done with the help of the following theorem.
	\begin{theorem}[{{\cite[Theorem 1.1]{logconc}}}]
		\label{log_concave_measure}
		Let $\psi$ be a log-concave probability measure on $\mathbb{R}^n$. For all measurable sets $A \subset \mathbb{R}^n$, for every point $x_0 \in \mathbb{R}^n$ and every number $r > 0$,
		\begin{multline*}
			\psi^{+}(A) \geq \frac{1}{2r}\Big(\psi(A) \ln \frac{1}{\psi(A)} + (1 - \psi(A))\ln \frac{1}{1 - \psi(A)} \\ + \ln \psi(\{ |x - x_0| \leq r \})\Big)
		\end{multline*}
	\end{theorem}

	\noindent Convexity of $\Delta_n$ ensures that the normalized Lebesgue measure $\mu$ on it is log-concave.

	In the application of the above theorem to measure $\mu$ on $\Delta_n$ it is possible to set $x_0 = (0, \ldots, 0)$, even though $x_0 \notin \Delta_n$, since when $r > \frac{1}{\sqrt{n}}$ for $x_1 = (\frac{1}{n}, \ldots, \frac{1}{n})$ and $r^{\prime} = \sqrt{r^2 - \frac{1}{n}}$ we shall have
	\[
		\ln	\mu(\{ |x| \leq r \}) = \ln \mu(\{|x - x_1| \leq r^{\prime}\})
	\]
	\[
		\frac{1}{2 r^{\prime}} \geq \frac{1}{2r}
	\]

	One question that arises after examining the lower bound is: how do we estimate $\psi(\{ |x - x_0| \leq r \})$ -- the measure of the ball with center at $x_0$ and radius $r$? In terms of Lemma \ref{expon_into_delta} we have the following result

	\begin{theorem}[{{\cite[Theorem 2.2]{SZ}}}]
		\label{concetration}
		There are absolute positive constants $T, c$ such that for all $t > \frac{T}{\sqrt{n}}$, putting $X = (X_1, \ldots, X_n)$ and $S = X_1 + \ldots + X_n$,
		\[
			\operatorname{Pr}\left(\frac{\Vert X \Vert_2}{S} > t \right) \leq e^{-ctn}	
		\]
	\end{theorem}
	\noindent Again, because of the symmetry, we can restrict our attention to the positive orthant $\mathbb{R}_+^n$.

	By combining Theorem \ref{concetration} with Lemma \ref{expon_into_delta} we get
	\[
		\mu(\{ |x| \leq r \}) \geq 1 - e^{-cnr}
	\]
	for $r > \frac{T}{\sqrt{n}}$. Now note
	\[
		r > \frac{T}{\sqrt{n}} \Rightarrow cnr \geq c\sqrt{n}T \geq cT \Rightarrow e^{-cnr} \leq e^{-cT},
	\]
	which implies that for $r > \frac{T}{\sqrt{n}}$
	\[
		\ln \mu(\{ |x| \leq r \}) \geq \ln (1 - e^{-cnr}) \geq -C e^{-cnr}
	\]
	for some constant $C > 0$.

	If we assume that $\mu(A) < c^{\prime}$ for some constant $0 < c^{\prime} < 1$, then we will have
	\[
		(1 - \mu(A)) \ln \left(\frac{1}{1 - \mu(A)}\right) \geq C_1 \mu(A)
	\]
	for some constant $C_1 > 0$.

	We would like the sum of the two last terms of
	\begin{equation*}
			\mu(A) \ln \frac{1}{\mu(A)} + (1 - \mu(A))\ln \frac{1}{1 - \mu(A)} \\ + \ln \mu(\{ |x - x_0| \leq r \}) 
	\end{equation*}
	to be non-negative, for this the following will be sufficient
	\[
		C_1\mu(A) \geq C e^{-cnr},
	\]
	which could be rewritten as
	\[
		\ln \frac{C_1\mu(A)}{C} \geq -cnr \Leftrightarrow r \geq \frac{1}{cn} \ln \frac{C}{C_1 \mu(A)} = \frac{1}{cn} \left(\ln \frac{1}{\mu(A)} + \ln \frac{C}{C_1}\right)
	\]
	If we replace $\geq$ above with equality and apply Theorem \ref{log_concave_measure} with $x_0 = 0$, we will get
	\[
		\mu^+(A) \geq \frac{1}{2} cn \mu(A) \frac{\ln \frac{1}{\mu(A)}}{\ln \frac{1}{\mu(A)} + \ln \frac{C}{C_1}}
	\]
	The condition $r > \frac{T}{\sqrt{n}}$ means that we need
	\[
		\ln \frac{1}{\mu(A)} + \ln \frac{C}{C_1} > c T \sqrt{n}
	\]
	to apply Theorem \ref{log_concave_measure} here. And so these last inequalities imply

	\begin{proposition}
		\label{small_sets}
		There are universal constants $c_s, C > 0$ such that \footnote{here for the sake of simplicity we forget about the auxiliary constants $c^{\prime}, C, C_1$ introduced during the proof of this proposition}
		\[
			\mu^+(A) \geq c_s n \mu(A)
		\]
		for all subsets $A \subset \Delta_n$ with $\mu(A) < e^{-C \sqrt{n}}$.
	\end{proposition}

	Now that we have dealt with the case of <<small>> sets, we can proceed to the main argument.

	\begin{definition}
		The gradient modulus $\Vert \nabla f \Vert_2$ of a locally Lipschitz function $f$ is
		\[
			\Vert \nabla f(x) \Vert_2 = \limsup_{\Vert x - y \Vert_2 \to 0^+} \frac{\Vert f(x) - f(y) \Vert_2}{\Vert x - y \Vert_2}
		\]
	\end{definition}

	The next lemma would allow us to switch between the two equivalents of the isoperimetric problem.

	\begin{lemma}[{{\cite[Proposition A]{Sodin2008}}}]
		\label{iso_to_fun}
		Let $\mu$ be a probability measure, $0 < a < \frac{1}{2}$ and $b > 0$. The following are equivalent

			(a) $\mathcal{I}_{\mu}(a) \geq b$

			(b) for any locally Lipschitz function $\phi: \operatorname{supp}\mu \to [0; 1]$ such that $\mu\{\phi = 0\} \geq \frac{1}{2}$ and $\mu\{\phi = 1\} \geq a$,
			\[
				\int \Vert \nabla \phi \Vert_2 d\mu \geq b
			\]
	\end{lemma}
	\begin{remark}
		\label{local_lipschitz}
		In this theorem and its applications in this article we can replace <<locally Lipschitz>> in $(b)$ with just <<Lipschitz>>, since it would not affect the implication $(a) \Rightarrow (b)$, and in the proof(\cite{Sodin2008}) of the implication $(b) \Rightarrow (a)$ only Lipschitz functions $\phi\colon \operatorname{supp}\mu \to [0; 1]$ were considered.
	\end{remark}

	The idea of <<getting rid>> of unwanted parts of our space would be realized through the so-called cut-off functions. A cut-off function maps our space to $[0; 1]$, where $0$ corresponds to the regions we are getting rid of, $1$ -- to the regions we want to keep, and we would also need our function to take values in between $0$ and $1$ to ensure continuity.

	The following lemma shows how a cut-off function $h$ can be applied.

	\begin{lemma}[{{\cite[Lemma 2]{Sodin2008}}}]
		\label{apply_cutoff}
		If $k, h\colon \mathbb{R}^n \to [0; 1]$ are two locally Lipschitz functions, then
		\[
			\Vert \nabla k \Vert_2 \geq \Vert \nabla(kh) \Vert_2 - \Vert \nabla h \Vert_2
		\]
	\end{lemma}

	\noindent By Lemma \ref{iso_to_fun} we can speak about the isoperimetric problem in terms of the integral of gradient modulus $\Vert \nabla \phi \Vert_2$. Lemma \ref{apply_cutoff} would allow us to pass from $\phi$ to function $\phi \cdot h$, which vanishes on the unwanted region of our space, at the cost of an error term $\Vert \nabla h \Vert_2$

	\[
		\int \Vert \nabla \phi \Vert_2 d\mu \geq \int \Vert \nabla (\phi \cdot h) \Vert_2 d\mu - \int \Vert \nabla h \Vert_2 d\mu
	\]

	Recall that our upper bound on gradient modulus of mapping $T: \mathbb{R}_+^n \to \Delta_n$ is
	\[
		\Vert \nabla T \Vert_2 = \left\Vert \frac{\partial T_j(x)}{\partial x_i} \right\Vert_2 \leq \frac{1}{\Vert x \Vert_1}(1 + \sqrt{n}\Vert T(x) \Vert_2)
	\]
	So we are going to need two cut-off functions: one for parts of $\Delta_n$ that are too far from the origin will be of the form
	\[
		h_1\colon \mathbb{R}^n \to [0; 1], \quad h_1(x) = \max(0, \min(1, 2 - c_1 \sqrt{n} \Vert x \Vert_2)),
	\]
	and will take care of large values of $\Vert T(x) \Vert_2$; another for the region of $\mathbb{R}^n$ with low $\Vert x \Vert_1$ will be of the form
	\[
		h_2\colon \mathbb{R}^n \to [0; 1], \quad h_2(x) = \max(0, \min(1, c_2 n^{-1} \Vert x \Vert_1 - 1)).
	\]
	Constants $c_1$ and $c_2$ will be chosen later. Note that the two cut-off functions are meant for different domains: $h_1$ for $\Delta_n$ and $h_2$ for $\mathbb{R}_+^n$, but since both $\mathbb{R}_+^n$ and $\Delta_n$ lie inside $\mathbb{R}^n$ we choose $\mathbb{R}^n$ as their domain of definition.

	To employ Lemma \ref{apply_cutoff} in our argument we are going to need some estimates related to the arising error terms. In the following lemma some properties of our cut-off functions and their gradients will be established. 

	\begin{lemma}[corresponds to Lemma 3 from \cite{Sodin2008}]
		\label{cutoff_props}
		The cut-off function $h_1$ has the following properties
		\begin{equation}
			\label{h_1_1}
			h_1(x) = 1 \Leftrightarrow \Vert x \Vert_2 \leq \frac{1}{c_1 \sqrt{n}}
		\end{equation}
		\begin{equation}
			\label{h_1_2}
			h_1(x) = 0 \Leftrightarrow \Vert x \Vert_2 \geq \frac{2}{c_1 \sqrt{n}}
		\end{equation}
		\begin{equation}
			\label{h_1_grad}
			\Vert \nabla h_1 \Vert_2 \leq c_1 \sqrt{n}
		\end{equation}
		The cut-off function $h_2$ has the following properties
		\begin{equation}
			\label{h_2_1}		
			h_2(x) = 1 \Leftrightarrow \Vert x \Vert_1 \geq \frac{2}{c_2}n
		\end{equation}
		\begin{equation}
			\label{h_2_2}
			h_2(x) = 0 \Leftrightarrow \Vert x \Vert_1 \leq \frac{n}{c_2}
		\end{equation}
		\begin{equation}
			\label{h_2_grad}
			\Vert \nabla h_2 \Vert_2 \leq \frac{c_2}{\sqrt{n}}
		\end{equation}
	\end{lemma}
	\begin{proof}
		Properties (\ref{h_1_1}), (\ref{h_1_2}), (\ref{h_2_1}), (\ref{h_2_2}) immediately follow from the definition of our cut-off functions.

		By the triangle inequality we note that the gradient modulus of $\Vert x \Vert_2$ considered as a function from $\mathbb{R}_+^n$ to $\mathbb{R}_+$ is not greater than $1$, thus
		\[
			\Vert \nabla h_1 \Vert_2 \leq c_1 \sqrt{n}
		\]

		Inequalities
		\[
			\Vert x + \Delta x \Vert_1 \leq \Vert x \Vert_1 + \Vert \Delta x \Vert_1,
		\]
		\[
			\Vert \Delta x \Vert_1 \leq \sqrt{n} \Vert \Delta x \Vert_2
		\]
		imply that the gradient modulus of $\Vert x \Vert_1$ is not greater than $\sqrt{n}$, from which we derive
		\[
			\Vert \nabla h_2 \Vert_2 \leq \frac{c_2}{\sqrt{n}}
		\]
	\end{proof}

	\begin{lemma}[corresponds to Lemma 4 from \cite{Sodin2008}]
		\label{cutoff_integral}
		For $\alpha \geq 0$ we have
		\[
			\nu^n \{ \Vert x \Vert_1 \leq \alpha n \} \leq \frac{1}{\sqrt{2\pi n}} (\alpha e)^n
		\]
		And for every $\alpha > T$ by Lemma \ref{concetration} we have
		\[
			\mu \left\{ \Vert x \Vert_2 \geq \frac{\alpha}{\sqrt{n}} \right\} \leq e^{- \alpha c \sqrt{n}}
		\]
	\end{lemma}
	\begin{proof}
		Since the density of $\nu^n$ everywhere in $\mathbb{R}_+^n$ is not greater than $1$, we can bound $\nu^n \{ \Vert x \Vert_1 \leq \alpha n \}$ above by the volume of the region of $\mathbb{R}_+^n$ defined by $\Vert x \Vert_1 \leq \alpha n$, which is equal to
		\[
			\frac{1}{n!} (\alpha n)^n
		\]
		By Stirling's approximation
		\[
			n! \geq \sqrt{2\pi n} \left(\frac{n}{e}\right)^n e^{\frac{1}{12n + 1}}
		\]
		And so we arrive at
		\[
			\nu^n \{ \Vert x \Vert_1 \leq \alpha n \} \leq \frac{1}{n!} (\alpha n)^n \leq \frac{1}{\sqrt{2\pi n}} \left( \frac{e}{n} \right)^n (\alpha n)^n e^{-\frac{1}{12n + 1}} \leq \frac{1}{\sqrt{2\pi n}}(\alpha e)^n
		\]

		Note that $\frac{\alpha}{\sqrt{n}} > \frac{T}{\sqrt{n}}$, so by Theorem \ref{concetration}
		\[
			\mu \left\{ \Vert x \Vert_2 \geq \frac{\alpha}{\sqrt{n}} \right\} \leq e^{-c \frac{\alpha}{\sqrt{n}} n} = e^{-\alpha c \sqrt{n}}
		\] 
	\end{proof}

	Now we are ready to present the main argument.

	\begin{proposition}[corresponds to Proposition 2 from \cite{Sodin2008}]
		\label{big_sets}
		There is a universal constant $c_{b} > 0$ such that for all $e^{-C\sqrt{n}} \leq t < \frac{1}{2}$
		\[
			\mathcal{I}_{\mu}(t) \geq c_b n t.
		\]
	\end{proposition}
	\begin{proof}
		Pick $e^{-C \sqrt{n}} \leq a < \frac{1}{2}$. According to Lemma \ref{iso_to_fun} and Remark \ref{local_lipschitz} the problem of finding lower bounds on $\mathcal{I}_{\mu}(a)$ is equivalent to the estimation of
		\[
			\int_{\Delta_n} \Vert \nabla f \Vert_2 d\mu
		\]
		for a Lipschitz function $f\colon \Delta_n \to [0; 1]$ such that 
		\begin{equation}
			\label{fun_assumption}
			\mu\{ f = 0 \} \geq \frac{1}{2} \textrm{ and } \mu\{ f = 1 \} \geq a
		\end{equation}

		To <<get rid>> of the parts of $\Delta_n$ that are too far from the origin we can use our cut-off function $h_1$ and by Lemma \ref{apply_cutoff} we will get
		\begin{equation}
			\label{step1}
			\int_{\Delta_n} \Vert \nabla f \Vert_2 d\mu \geq \int_{\Delta_n} \Vert \nabla(fh_1) \Vert_2 d\mu - \int_{\Delta_n} \Vert \nabla h_1 \Vert_2 d\mu
		\end{equation}
		Here by Lemma \ref{cutoff_props} we can estimate the error term as
		\begin{equation}
			\label{step2}
			\int_{\Delta_n} \Vert \nabla h_1 \Vert_2 d\mu \leq c_1 \sqrt{n} \, \mu\left\{ \Vert x \Vert_2 \geq \frac{1}{c_1 \sqrt{n}} \right\}
		\end{equation}

		Mapping $T\colon \mathbb{R}_+^n \to \Delta_n$ transforms measure $\nu^n$ into $\mu$, which allows us to replace integrals over $\Delta_n$ with integrals over $\mathbb{R}_+^n$ as follows
		\[
			\int_{\Delta_n} w\, d\mu = \int_{\mathbb{R}_+^n} (w \circ T) d\nu^n
		\]
		Denote $(f h_1) \circ T$ by $g$. As we already noted
		\begin{equation}
			\label{step3}
			\int_{\Delta_n} \Vert \nabla (f h_1) \Vert_2 d\mu  = \int_{\mathbb{R}^{n}_{+}} \Vert \nabla (f h_1) \circ T \Vert_2 d\nu^n
		\end{equation}
		
		An observation similar to the chain rule of differentiation could be made
		\[
			\Vert \nabla (a \circ b) \Vert_2 \leq \Vert (\nabla a) \circ b \Vert_2 \cdot \Vert \nabla b \Vert_2,
		\]
		which in our case would mean that
		\begin{equation}
			\label{step4}
			\int_{\mathbb{R}^{n}_{+}} \Vert \nabla (f h_1) \circ T \Vert_2 d\nu^n \geq \int_{\mathbb{R}^n_{+}} \frac{\Vert \nabla g \Vert_2}{\Vert \nabla T \Vert_2} d \nu^n,
		\end{equation}
		and since $\Vert \nabla T \Vert_2 \neq 0$ by Lemma \ref{lipschitz_constant} could be bounded above by $\frac{1}{\Vert x \Vert_1}(1 + \sqrt{n}\Vert T(x) \Vert_2)$ we have
		\begin{equation}
			\label{step5}
			\int_{\mathbb{R}^n_{+}} \frac{\Vert \nabla g \Vert_2}{\Vert \nabla T \Vert_2} d \nu^n \geq \int_{\mathbb{R}^n_{+}} \frac{\Vert \nabla g \Vert_2 \Vert x \Vert_1}{1 + \sqrt{n}\Vert T(x) \Vert_2} d \nu^n
		\end{equation}
		But $h_1$ is zero when $\Vert x \Vert_2 \geq \frac{2}{c_1 \sqrt{n}}$ by Lemma \ref{cutoff_props}. So the gradient modulus $\Vert \nabla g \Vert_2$ is equal to zero when $\Vert T(x) \Vert_2 > \frac{2}{c_1 \sqrt{n}}$. And because of this,
		\begin{multline}
			\label{step6}
			\int_{\mathbb{R}^n_{+}} \frac{\Vert \nabla g \Vert_2 \Vert x \Vert_1}{1 + \sqrt{n}\Vert T(x) \Vert_2} d \nu^n \geq \int_{\mathbb{R}^n_{+}} \frac{\Vert \nabla g \Vert_2 \Vert x \Vert_1}{1 + \sqrt{n} \frac{2}{c_1 \sqrt{n}}} d \nu^n \\ = \frac{1}{1 + \frac{2}{c_1}}\int_{\mathbb{R}^n_{+}} \Vert \nabla g \Vert_2 \Vert x \Vert_1 d\nu^n
		\end{multline}

		Now to <<get rid>> of the region of $\mathbb{R}_+^n$ where $\Vert x \Vert_1$ is too small we will apply our cut-off function $h_2$ and by Lemma \ref{apply_cutoff} get
		\begin{equation}
			\label{step7}
			\int_{\mathbb{R}^n_{+}} \Vert \nabla g \Vert_2 \Vert x \Vert_1 d\nu^n \geq \int_{\mathbb{R}^n_{+}} \Vert \nabla (g h_2) \Vert_2 \Vert x \Vert_1 d\nu^n - \int_{\mathbb{R}^n_{+}} \Vert \nabla h_2 \Vert_2 \Vert x \Vert_1 d\nu^n
		\end{equation}
		By Lemma \ref{cutoff_props} we have the following upper bound on the error term
		\begin{multline}
			\label{step8}
			\int_{\mathbb{R}^n_{+}} \Vert \nabla h_2 \Vert_2 \Vert x \Vert_1 d\nu^n \leq \frac{c_2}{\sqrt{n}} \frac{2n}{c_2} \nu^n \left\{ \Vert x \Vert_1 \leq \frac{2n}{c_2} \right\} \\ = 2 \sqrt{n} \nu^n \left\{ \Vert x \Vert_1 \leq \frac{2n}{c_2} \right\}
		\end{multline}
		Cut-off function $h_2$ is zero when $\Vert x \Vert_1 \leq \frac{n}{c_2}$, thus $\Vert \nabla (g h_2) \Vert_2$ is equal to zero when $\Vert x \Vert_1 < \frac{n}{c_2}$, from which it follows that
		\begin{equation}
			\label{step9}
			\int_{\mathbb{R}^n_{+}} \Vert \nabla (g h_2) \Vert_2 \Vert x \Vert_1 d\nu^n \geq \frac{n}{c_2} \int_{\mathbb{R}^n_{+}} \Vert \nabla (g h_2) \Vert_2 d\nu^n 
		\end{equation}

		Now consider function $g h_2\colon \mathbb{R}_+^n \to [0; 1]$
		\[
			g h_2 = ((f \cdot h_1) \circ T) \cdot h_2
		\]
		Note that if $(f \circ T)(x) = 0$ for $x \in \mathbb{R}_+^n$, then $(g h_2)(x) = 0$ too. By our assumption $\mu\{ f = 0 \} \geq \frac{1}{2}$, which implies $\nu^n\{ g h_2 = 0 \} \geq \frac{1}{2}$. Function $gh_2$ equals to $1$ at a point $x \in \mathbb{R}_+^n$ if and only if
		\[
			(f \circ T)(x) = 1, \textrm{ and } (h_1 \circ T)(x) = 1, \textrm{ and } h_2(x) = 1.
		\]
		To estimate $\nu^n\{ gh_2 = 1 \}$ we will subtract $\nu^n\{ (h_1 \circ T) < 1 \} = \mu \{ h_1 < 1 \}$ and $\nu^n \{ h_2 < 1 \}$ from $\nu^n\{ (f \circ T) = 1\} = \mu \{ f = 1 \}$, which by our assumption (\ref{fun_assumption}) is greater than $a$, and get
		\[
			\nu^n\{ gh_2 = 1 \} \geq a - \mu \{ h_1 < 1 \} - \nu^n \{ h_2 < 1 \}
		\]
		By Lemma \ref{cutoff_props}
		\[
			\mu \{ h_1 < 1 \} = \mu \left\{ \Vert x \Vert_2 > \frac{1}{c_1 \sqrt{n}} \right\}
		\]
		\[
			\nu^n \{ h_2 < 1 \} = \nu^n \left\{ \Vert x \Vert_1 < \frac{2n}{c_2} \right\}
		\]
		Isoperimetric inequality (\ref{nu_isoperimetric}) on $\nu^n$ combined with Lemma \ref{iso_to_fun} would give us
		\begin{multline}
			\label{nu_fun}
			\int_{\mathbb{R}_+^n} \Vert \nabla(g h_2) \Vert_2 d\nu^n \geq \frac{1}{2\sqrt{6}}\bigg( a \\ - \mu \left\{ \Vert x \Vert_2 > \frac{1}{c_1 \sqrt{n}} \right\} - \nu^n \left\{ \Vert x \Vert_1 < \frac{2n}{c_2} \right\} \bigg)
		\end{multline}

		Putting inequalities (\ref{step1}), (\ref{step2}), (\ref{step3}), (\ref{step4}), (\ref{step5}), (\ref{step6}), (\ref{step7}), (\ref{step8}), (\ref{step9}) together, we arrive at
		\begin{multline*}		
			\int_{\Delta_n} \Vert \nabla f \Vert_2 d\mu \geq \frac{1}{c_2} \frac{1}{1 + \frac{2}{c_1}} n \int_{\mathbb{R}^n_{+}} \Vert \nabla(gh_2) \Vert_2 d\nu^n \\ - c_1 \sqrt{n} \mu\left\{ \Vert x \Vert_2 \geq \frac{1}{c_1 \sqrt{n}} \right\} - \frac{2}{1 + \frac{2}{c_1}} \sqrt{n} \nu^n \left\{ \Vert x \Vert_1 \leq \frac{2n}{c_2} \right\}
		\end{multline*}
		We combine this with inequality (\ref{nu_fun}) and get
		\begin{multline*}
			\int_{\Delta_n} \Vert \nabla f \Vert_2 d\mu \geq \frac{1}{2\sqrt{6}} \frac{1}{c_2} \frac{1}{1 + \frac{2}{c_1}} n a \\ - \left( \frac{1}{2\sqrt{6}} \frac{1}{c_2} \frac{1}{1 + \frac{2}{c_1}} n + c_1 \sqrt{n}  \right) \mu\left\{ \Vert x \Vert_2 \geq \frac{1}{c_1 \sqrt{n}} \right\} \\ - \left( \frac{1}{2\sqrt{6}} \frac{1}{c_2} \frac{1}{1 + \frac{2}{c_1}} n + \frac{2}{1 + \frac{2}{c_1}} \sqrt{n} \right) \nu^n \left\{ \Vert x \Vert_1 \leq \frac{2n}{c_2} \right\},
		\end{multline*}
		which could be rewritten as
		\begin{multline*}
			\int_{\Delta_n} \Vert \nabla f \Vert_2 d\mu \geq \frac{1}{2\sqrt{6}} \frac{1}{c_2} \frac{1}{1 + \frac{2}{c_1}} n \Bigg( a \\ - \left( 1 + 2\sqrt{6} c_1 c_2 \left(1 + \frac{2}{c_1}\right) \frac{1}{\sqrt{n}}\right) \mu\left\{ \Vert x \Vert_2 \geq \frac{1}{c_1 \sqrt{n}} \right\} \\ - \left( 1 + 4\sqrt{6} c_2 \frac{1}{\sqrt{n}} \right)\nu^n \left\{ \Vert x \Vert_1 \leq \frac{2n}{c_2} \right\} \Bigg)
		\end{multline*}
		
		If $\frac{1}{c_1} > T \Leftrightarrow c_1 < \frac{1}{T}$, then by Lemma \ref{cutoff_integral} we should have
		\begin{multline*}
			\int_{\Delta_n} \Vert \nabla f \Vert_2 d\mu \geq \frac{1}{2\sqrt{6}} \frac{1}{c_2} \frac{1}{1 + \frac{2}{c_1}} n \Bigg( a \\ - \left( 1 + 2\sqrt{6} c_2 \left(c_1 + 2\right) \frac{1}{\sqrt{n}}\right) e^{- \frac{c}{c_1}\sqrt{n}} \\ - \left( 1 + 4\sqrt{6} c_2 \frac{1}{\sqrt{n}} \right) \frac{1}{\sqrt{2\pi n}} \left( \frac{2e}{c_2} \right)^n \Bigg)
		\end{multline*}
		Now we can choose appropriate values for constants $c_1$ and $c_2$. We choose $c_2$ to be large enough so that
		\[
			\left( 1 + 4\sqrt{6} c_2 \frac{1}{\sqrt{n}} \right) \frac{1}{\sqrt{2\pi n}} \left( \frac{2e}{c_2} \right)^n \leq \frac{1}{3} e^{-C \sqrt{n}}
		\]
		holds for all natural $n$. This is possible since one can note that
		\[
			\left( \frac{2e}{c_2} \right)^n = e^{- \ln \left(\frac{2e}{c_2}\right) n}
		\]
		After that we choose $c_1$ to be small enough so that
		\[
			\left( 1 + 2\sqrt{6} c_2 \left(c_1 + 2\right) \frac{1}{\sqrt{n}}\right) e^{- \frac{c}{c_1}\sqrt{n}} \leq \frac{1}{3} e^{-C \sqrt{n}}
		\]
		holds for all natural $n$.

		Our $a$ is at least $e^{-C \sqrt{n}}$, which means
		\begin{multline*}
			\int_{\Delta_n} \Vert \nabla f \Vert_2 d\mu \geq \frac{1}{2\sqrt{6}} \frac{1}{c_2} \frac{1}{1 + \frac{2}{c_1}} n \left( a - \frac{2}{3}e^{-C \sqrt{n}} \right) \geq \frac{1}{2\sqrt{6}} \frac{1}{c_2} \frac{1}{1 + \frac{2}{c_1}} n \left( \frac{1}{3} a \right)
		\end{multline*}
		And since $f$ here can be an arbitrary Lipschitz function $f\colon \Delta_n \to [0; 1]$ with
		\[
			\mu\{ f = 0 \} \geq \frac{1}{2} \textrm{ and } \mu\{ f = 1 \} \geq a
		\]
		we by Lemma \ref{iso_to_fun} conclude
		\[
			\mathcal{I}_{\mu}(a) \geq \frac{1}{6 \sqrt{6}} \frac{1}{c_2} \frac{1}{1 + \frac{2}{c_1}} n a
		\]
	\end{proof}

	Propositions \ref{small_sets} and \ref{big_sets} imply
	\begin{theorem}
		\label{lambda_iso}
		For the Lebesgue measure $\lambda$ on the unit-volume simplex $\omega_n \Delta_n$ the following isoperimetric inequality
		\[
			\mathcal{I}_{\lambda}(t) \geq c_{\lambda} t
		\]
		holds for all $t \in (0; \frac{1}{2})$, where $c_{\lambda} > 0$ is a universal constant independent of the dimension $n$.
	\end{theorem}
	\begin{proof}
		By Proposition \ref{small_sets}
		\[
			\mu^+(A) \geq c_s n \mu(A)
		\]
		for all $A \subset \Delta_n$ with $\mu(A) \in (0; e^{-C \sqrt{n}})$, and by Proposition \ref{big_sets}
		\[
			\mathcal{I}_{\mu}(t) \geq c_b n t
		\]
		for all $t \in [e^{-C \sqrt{n}}; \frac{1}{2})$, which means that
		\[
			\mathcal{I}_{\mu}(t) \geq \min(c_s, c_b) n t
		\]
		for all $t \in (0; \frac{1}{2})$.

		Equation (\ref{mu_lambda}) relates $\mathcal{I}_{\mu}$ and $\mathcal{I}_{\lambda}$ to each other as
		\[
			\mathcal{I}_{\lambda} = \frac{1}{\omega_n} \mathcal{I}_{\mu}
		\]
		Thus for all $t \in (0; \frac{1}{2})$ we must have
		\[
			\mathcal{I}_{\lambda}(t) \geq \min(c_s, c_b) \frac{n}{\omega_n} t.
		\]

		Here we could note that $\frac{n}{\omega_n}$ is positive for all $n$ and that by Stirling's approximation
		\[
			\lim_{n \to \infty} \frac{n}{\omega_n} = e,
		\]
		which must imply that
		\[
			\inf_{n} \frac{n}{\omega_n} > 0
		\]
		So we can take
		\[
			\min(c_s, c_b) \inf_{n} \frac{n}{\omega_n}
		\]
		as our constant $c_{\lambda}$.
	\end{proof}

	From this isoperimetric inequality we conclude
	\begin{theorem}
		\label{simplex_upper}
		Inside a unit-volume simplex $\omega_n \Delta_n$ two bodies $A$ and $B$ of volume $\varepsilon \in (0; \frac{1}{2})$ are at a distance at most
		\[
			- c \ln \varepsilon 
		\]
		for some universal constant $c > 0$ independent of the dimension $n$ and volume $\varepsilon$.
	\end{theorem}
	\begin{proof}
		We are interested in the least values $\delta_A, \delta_B$ such that the $\delta_A$-enlargement of body $A$ in $\omega_n \Delta_n$ will be of volume $\frac{1}{2}$ and the $\delta_B$-enlargement of body $B$ will be of volume $\frac{1}{2}$ too. For these enlargements we shall have
		\[
			\operatorname{dist}(A_{\delta_A}, B_{\delta_B}) = 0,
		\]
		from which
		\[
			\operatorname{dist}(A, B) \leq \delta_A + \delta_B
		\]
		follows.

		Isoperimetric inequality from Theorem \ref{lambda_iso} provides an estimate on the growth of $\delta$-enlargements of our bodies
		\[
			\delta_+ \lambda(A_{\delta}) \geq c_{\lambda} \lambda(A_{\delta})
		\]
		which holds as long as $\lambda(A_{\delta}) < \frac{1}{2}$.

		And so to bound $\delta_A$ above we would like consider a function $y(\delta)$ that behaves in accordance with our lower bound
		\begin{equation}
			\label{init_lambda}
			y(0) = \lambda(A)
		\end{equation}
		\begin{equation}
			\label{diff_eq_lambda}
			y^{\prime} = c_{\lambda} y
		\end{equation}
		If by $\delta_{M}$ we will denote the moment when $y$ reaches $\frac{1}{2}$, i.\,e. $y(\delta_M) = \frac{1}{2}$, then $\delta_A \leq \delta_M$. Indeed, otherwise $\delta_M < \delta_A$, but functions $\lambda(A_{\delta})$ and $y(\delta)$ coincide at $\delta = 0$ and for all $\delta \in [0; \delta_M]$ we should have
		\[
			\lambda(A_{\delta}) \geq y(\delta) \textrm{ and } \delta_+ \lambda(A_{\delta}) \geq c_{\lambda} \lambda(A_{\delta}) \geq c_{\lambda} y(\delta) = \delta_+ y(\delta)
		\]
		And so we reach contradiction
		\[
			\frac{1}{2} = \lambda(A_{\delta_A}) > \lambda (A_{\delta_M}) \geq y(\delta_M) = \frac{1}{2}
		\]

		A solution to differential equation (\ref{diff_eq_lambda}) should be of the form
		\[
			C e^{c_{\lambda} \delta}
		\]
		and since at $\delta = 0$ by our initial condition (\ref{init_lambda}) we should have $y(0) = \lambda(A)$ we reach conclusion
		\[
			y(\delta) = \lambda(A) e^{c_{\lambda} \delta}
		\]
		So $\delta_M$ will be a solution to equation
		\[
			\lambda(A) e^{c_{\lambda} \delta_M} = \frac{1}{2},
		\]
		which after taking logarithm on both sides turns into
		\[
			\ln \lambda(A) + c_{\lambda} \delta_M = - \ln 2
		\]
		\[
			\delta_M = -\frac{1}{c_{\lambda}}\left(\ln \lambda(A) + \ln 2 \right)
		\]

		By the same reasoning $\delta_B \leq \delta_M$, and thus
		\[
			\operatorname{dist}(A, B) \leq \delta_A + \delta_B \leq - \frac{2}{c_{\lambda}}(\ln \lambda(A) + \ln 2) \leq -\frac{2}{c_{\lambda}} \ln \lambda(A)
		\]
	\end{proof}

	\newpage

	\subsection{$\ell_p$-balls.}

	By the $\ell_p^n$ unit ball we mean
	\[
		\ell_p^n = \{ (x_1, \ldots, x_n) \in \mathbb{R}^n \mid |x_1|^p + \ldots + |x_n|^p \leq 1 \}
	\]

	Let $\mu$ be a normalized Lebesgue measure on it. Note that $\mu(\ell_p^n) = 1$. The volume of $\ell_p^n$ is equal to

	\[
		2^n \frac{\Gamma\left(1 + \frac{1}{p}\right)^n}{\Gamma\left(1 + \frac{n}{p}\right)}
	\]
	by theorem $1$ from \cite{vol_p}. So in order to get a unit-volume $\ell_p^n$ ball we would need to stretch the $\ell_p^n$ unit ball by a factor of
	\[
		\omega_n = \frac{\Gamma\left(1 + \frac{n}{p}\right)^{\frac{1}{n}}}{2 \Gamma\left(1 + \frac{1}{p}\right)} \sim \frac{n^{\frac{1}{p}}}{2 \Gamma\left(1 + \frac{1}{p}\right) (pe)^{\frac{1}{p}}}
	\]
	By $\lambda$ denote the Lebesgue measure on $\omega_n \ell_p^n$. Yet again by (\ref{mu_lambda}) we should have proportionality of the isoperimetric functions
	\begin{equation}
		\label{ell_p_iso_relation}
		\mathcal{I}_{\mu} = \omega_n \mathcal{I}_{\lambda}
	\end{equation}

	The following theorem was proven by Sasha Sodin in \cite{Sodin2008}.

	\begin{theorem}[{{\cite[Theorem 1]{Sodin2008}}}]
		There exists a universal constant $c > 0$ such that for $1 \leq p \leq 2$, $0 < a < \frac{1}{2}$
		\[
			\mathcal{I}_{\mu}(a) \geq c n^{\frac{1}{p}} a \log^{1 - \frac{1}{p}} a
		\]
	\end{theorem}

	It follows that
	\begin{theorem}
		\label{ell_iso}
		For every $p \in [1; 2]$ there exists a positive constant $c_p > 0$ such that
		\[
			\mathcal{I}_{\lambda}(a) > c_p a \log^{1 - \frac{1}{p}} a
		\]
	\end{theorem}
	\begin{proof}
		By (\ref{ell_p_iso_relation}) we already now that
		\[
			\mathcal{I}_{\lambda}(a) \geq c \frac{1}{\omega_n} n^{\frac{1}{p}} a \log^{1 - \frac{1}{p}} a
		\]
		The number $\frac{c}{\omega_n} n^{\frac{1}{p}}$ is positive for all $n$ and by Stirling's approximation
		\[
			\lim_{n \to \infty} \frac{c}{\omega_n} n^{\frac{1}{p}} = 2c\Gamma\left(1 + \frac{1}{p}\right) (pe)^{\frac{1}{p}} > 0
		\]
		Thus
		\[
			c_p = \inf_{n} c \frac{1}{\omega_n} n^{\frac{1}{p}}  > 0
		\]
		and
		\[
			\mathcal{I}_{\lambda}(a) \geq c_p a \log^{1 - \frac{1}{p}} a
		\]
	\end{proof}

	From this isoperimetric inequality we derive
	\begin{theorem}
		\label{ell_p_upper}
		Inside a unit-volume $\ell_n^p$ ball $\omega_n \ell_p^n$ two bodies $A$ and $B$ of volume $\varepsilon \in (0; \frac{1}{2})$ are at a distance at most
		\[
			C_p \log^{\frac{1}{p}} \frac{1}{\varepsilon}
		\]
		for some constant $C_p > 0$ independent of dimension $n$ and volume $\varepsilon$. 
	\end{theorem}
	\begin{proof}
		We are interested in the least values $\delta_A, \delta_B$ such that the $\delta_A$-enlargement of body $A$ in $\omega_n \ell_p^n$ will be of volume $\frac{1}{2}$ and the $\delta_B$-enlargement of body $B$ will be of volume $\frac{1}{2}$ too. For these enlargements we shall have
		\[
			\operatorname{dist}(A_{\delta_A}, B_{\delta_B}) = 0,
		\]
		from which
		\[
			\operatorname{dist}(A, B) \leq \delta_A + \delta_B
		\]
		follows.

		The isoperimetric inequality from Theorem \ref{ell_iso} allows to estimate the growth of $\lambda(A_{\delta})$ as
		\[
			\delta_+ \lambda(A_{\delta}) \geq c_p \lambda(A_{\delta}) \log^{1 - \frac{1}{p}} \frac{1}{\lambda(A_{\delta})}
		\]
		while $\lambda(A_{\delta}) < \frac{1}{2}$.

		So we would like to consider a function $y(\delta)$ that behaves in accordance with our lower bound
		\begin{equation}
			\label{init_cond_ell}
			y(0) = \varepsilon
		\end{equation}
		\begin{equation}
			\label{diff_eq_ell}
			y^{\prime} = c_p y \log^{1 - \frac{1}{p}} \frac{1}{y}
		\end{equation}
		If by $\delta_M$ we will denote the moment when $y$ reaches one half, i.\,e. $y(\delta_M) = \frac{1}{2}$, then $\delta_A \leq \delta_M$. Indeed, otherwise $\delta_M < \delta_A$, but functions $\lambda(A_{\delta})$ and $y(\delta)$ coincide at $\delta = 0$ and for all $\delta \in [0; \delta_M]$ we should have
		\begin{multline*}
			\lambda(A_{\delta}) \geq y(\delta) \textrm{ and } \delta_+ \lambda(A_{\delta}) \geq c_p \lambda(A_{\delta}) \log^{1 - \frac{1}{p}} \frac{1}{\lambda(A_{\delta})} \\ \geq c_p y(\delta) \log^{1 - \frac{1}{p}} \frac{1}{y(\delta)} = \delta_+ y(\delta),
		\end{multline*}
		since $x (-\log x)^{1 - \frac{1}{p}}$ is increasing on $(0; \frac{1}{2}]$(see Appendix \ref{weird_func_inc}). And so we reach a contradiction
		\[
			\frac{1}{2} = \mu(A_{\delta_A}) > \mu(A_{\delta_M}) \geq y(\delta_M) = \frac{1}{2}
		\]

		Differential equation (\ref{diff_eq_ell}) is separable
		\[
			dy = c_p y (-\log y)^{1 - \frac{1}{p}} d\delta
		\]
		\[
			- (-\log y)^{\frac{1}{p} - 1} \left(-\frac{1}{y} dy\right) = c_p d\delta
		\]
		\[
			-\int (-\log y)^{\frac{1}{p} - 1} d(-\log y) = \int c_p d\delta
		\]
		\[
			- p (-\log y)^{\frac{1}{p}} = c_p \delta + C_0
		\]
		Our initial condition (\ref{init_cond_ell}) gives us
		\[
			- p (-\log \varepsilon)^{\frac{1}{p}} = C_0
		\]
		And for $\delta = \delta_M$ we should have
		\[
			- p (\log 2)^{\frac{1}{p}} = c_p \delta_M - p (-\log \varepsilon)^{\frac{1}{p}}
		\]
		\[
			\delta_M = \frac{1}{c_p}\left( p (-\log \varepsilon)^{\frac{1}{p}} - p (\log 2)^{\frac{1}{p}} \right) \leq \frac{p}{c_p} (-\log \varepsilon)^{\frac{1}{p}}
		\]

		By the same reasoning $\delta_B \leq \delta_M$, and we conclude
		\[
			\operatorname{dist}(A, B) \leq \delta_A + \delta_B \leq 2 \delta_M \leq \frac{2p}{c_p} (-\log \varepsilon)^{\frac{1}{p}}
		\]
	\end{proof}

	\newpage

	\section{Lower bounds.}
	\label{lower_bounds}

	\textbf{Introduction.} Here we are going to be concerned with the lower bounds on the largest distance between two subsets of volume $0 < \varepsilon < \frac{1}{2}$. We will derive the lower bounds simply by considering certain hyperplane cuts of our convex bodies. For families of convex bodies such as the euclidean balls, cubes, hyperoctahedrons, simplexes and $\ell_p$ balls specific lower bounds will be shown in Theorems \ref{ball_sup}, \ref{cubes}, \ref{simplex}, \ref{l_p_balls}. It turns out that for euclidean balls o{}ur lower bounds coincide with the upper bounds(see Theorem \ref{ball_exact}). In Theorem \ref{general_sup} a general lower bound will be established, showing that in a way the family of euclidean balls is optimal in regard to our problem.

	\medskip

	It was already shown that for unit-volume cube, ball, simplex and $\ell_p$ balls with $p \in [1; 2]$ the largest distance is bounded above by some constant dependent on $\varepsilon$ but not on the dimension $n$. That is why it makes sense to consider the lower bounds on the supremum of all possible distances between two subsets of volume $\varepsilon$ that take place as $n$ tends to infinity.

	For a family of convex bodies $K_n$ by $d_n(\varepsilon)$ here we denote the supremum of all possible distances between two subsets of volume $\varepsilon \in (0; \frac{1}{2})$ in $K_n$.

	\begin{theorem}
		\label{ball_sup}
		When $K_n$ are the unit-volume euclidean balls we have
		\[
			\liminf_{n \to \infty} d_n(\varepsilon) \geq - 2 \frac{1}{\sqrt{e}} \Phi^{-1}(\varepsilon)
		\]
		The function $- 2 \frac{1}{\sqrt{e}} \Phi^{-1}(\varepsilon)$ is asymptotically equivalent to
		\[
			-2\frac{1}{\sqrt{\pi e}} \sqrt{-\ln \varepsilon}
		\]
		as $\varepsilon \to 0$.
	\end{theorem}
	\begin{proof}
		In the unit-volume $n$-ball $\omega_n B^n \subset \mathbb{R}^n$, where the radii is
		\[
			\omega_n = \frac{\Gamma\left(\frac{n}{2} + 1\right)^{\frac{1}{n}}}{\sqrt{\pi}} \sim \sqrt{\frac{n}{2\pi e}},
		\]
		consider the diagonal from $(-\omega_n, 0, \ldots, 0)$ to $(\omega_n, 0, \ldots, 0)$, i.\,e. a diagonal corresponding to the $X_1$-axis. We will be interested in the hyperplanes orthogonal to this diagonal, i.\,e. hyperplanes defined by $X_1 = t$.

		Pick a number $a$ such that
		\[
			\frac{1}{\sqrt{e}} \Phi^{-1}(\varepsilon) < -a < 0
		\]
		If we consider the uniform probability distribution on $\omega_n B^n$, then we could think of $X_1$ as of a random variable. We would like to consider the part of our ball that corresponds to $X_1 \leq -a$. The volume would be equal to
		\[
			\Pr(X_1 \leq -a) = \Pr(\sqrt{n} \omega_n^{-1} X_1 \leq -\sqrt{n} \omega_n^{-1} a)
		\]

		By theorem 1 of \cite{ball_section} as $n$ tends to infinity the distribution of $n^{\frac{1}{2}} \omega_n^{-1} X_1$ converges in total variation to the standard normal distribution on $\mathbb{R}$, whose probability density function is
		\[
			\frac{1}{\sqrt{2\pi}} e^{-\frac{1}{2}x^2}
		\]

		Furthermore, note that
		\[
			\lim_{n \to \infty} -\sqrt{n} \omega_n^{-1} a = \lim_{n \to \infty} -\sqrt{n} \frac{\sqrt{2\pi e}}{\sqrt{n}} a = -\sqrt{2 \pi e} a > \sqrt{2\pi} \Phi^{-1}(\varepsilon)
		\]
		So for all sufficiently large $n$ we shall have
		\begin{multline*}
			\sqrt{2\pi} \Phi^{-1}(\varepsilon) + \delta < -\sqrt{n} \omega_n^{-1} a \\ \Rightarrow \Pr(\sqrt{n} \omega_n^{-1} X_1 \leq -\sqrt{n} \omega_n^{-1} a) \geq \Pr(\sqrt{n} \omega_n^{-1} X_1 \leq \sqrt{2\pi} \Phi^{-1}(\varepsilon) + \delta)
		\end{multline*}
		for some $\delta > 0$. Because distribution of $n^{\frac{1}{2}} \omega_n^{-1} X_1$ converges in total variation to the standard normal distribution, for all sufficiently large $n$ we have
		\begin{multline*}
			\Pr(X_1 \leq -a) \geq \Pr(\sqrt{n} \omega_n^{-1} X_1 \leq \sqrt{2\pi} \Phi^{-1}(\varepsilon) + \delta) \\ \geq \frac{1}{\sqrt{2\pi}} \int_{-\infty}^{\sqrt{2\pi} \Phi^{-1}(\varepsilon)} e^{-\frac{1}{2}x^2} dx = \int_{-\infty}^{\sqrt{2\pi} \Phi^{-1}(\varepsilon)} e^{-\pi \left(\frac{1}{\sqrt{2\pi}}x\right)^2} d\left(\frac{1}{\sqrt{2\pi}} x \right) \\ = \int_{-\infty}^{\Phi^{-1}(\varepsilon)} e^{-\pi x^2} dx = \varepsilon
		\end{multline*}

		By symmetry we have a similar result for the part of our ball defined by $X_1 \geq a$. That means that for sufficiently large $n$ we are going to have two subsets of volume at least $\varepsilon$ at a distance $2a$. But $a$ was chosen as an arbitrary number lesser than
		\[
			-\frac{1}{\sqrt{e}} \Phi^{-1}(\varepsilon),
		\]
		from which the statement of the theorem follows.
	\end{proof}

	By combining this with Theorem \ref{euc_ball_dist} we get
	\begin{theorem}
		\label{ball_exact}
		When $K_n$ are the unit-volume euclidean balls
		\[
			\lim_{n \to \infty} d_n(\varepsilon) = -2\frac{1}{\sqrt{e}} \Phi^{-1}(\varepsilon)
		\]
	\end{theorem}

	\begin{theorem}
		\label{cubes}
		When $K_n$ are the unit cubes we have
		\[
			\liminf_{n \to \infty} d_n(\varepsilon) \geq -2\sqrt{\frac{\pi}{6}} \Phi^{-1}(\varepsilon)
		\]
		The function $-2\sqrt{\frac{\pi}{6}} \Phi^{-1}(\varepsilon)$ is asymptotically equivalent to
		\[
			\frac{2}{\sqrt{6}}\sqrt{- \ln \varepsilon}
		\]
		as $\varepsilon \to 0$. 
	\end{theorem}
	\begin{proof}
		By the main diagonal of a cube $(0; 1)^n$ we mean a segment from the origin to $(1, \ldots, 1)$. We would be considering hyperplanes orthogonal to the main diagonal. 

		For each point on the main diagonal we could consider the area of the corresponding orthogonal hyperplane section of $(0; 1)^n$. This gives rise to a probability distribution on the segment from $(0, \ldots, 0)$ to $(1, \ldots, 1)$, whose length is $\sqrt{n}$. We will take the midpoint of this segment as the origin, i.\,e. we have a probability distribution on $[-\frac{\sqrt{n}}{2}; +\frac{\sqrt{n}}{2}]$.

		The variance of the uniform distribution on the segment $[0; 1]$ is equal to
		\[
			\sigma^2 = \int_{0}^{1} \left(x - \frac{1}{2}\right)^2 dx = 2 \int_{0}^{\frac{1}{2}} x^2 dx = 2 \frac{1}{3} \frac{1}{2^3} = \frac{1}{12}
		\]
		And our distribution on $[-\frac{\sqrt{n}}{2}; +\frac{\sqrt{n}}{2}]$ could be produced by $n$ random variables $X_1, \ldots, X_n$ uniformly distributed on $[0; 1]$ as
		\[
			\sqrt{n}\left(\frac{X_1 + \ldots + X_n}{n} - \frac{1}{2} \right)
		\]
		So by central limit theorem as $n$ goes to infinity our distribution converges to a normal distribution $\mathcal{N}(0, \sigma^2)$, whose probability density function would be
		\[
			\frac{1}{\sigma \sqrt{2\pi}} e^{-\frac{1}{2}\frac{x^2}{\sigma^2}} = \sqrt{\frac{6}{\pi}}e^{-6 x^2}
		\]

		Pick a number $a$ such that
		\[
			\sqrt{\frac{\pi}{6}} \Phi^{-1}(\varepsilon) < -a < 0
		\]
		And consider the part of the unit cube $(0; 1)^n$ whose orthogonal projection on the main diagonal lies inside $[0; \frac{1}{2}\sqrt{n} - a]$, i.\,e. a certain hyperplane section. As $n$ goes to infinity the volume of this region would converge to the value of the cumulative distribution function of $\mathcal{N}(0, \sigma^2)$ at $-a$, which is equal to
		\begin{multline*}
			\int_{-\infty}^{-a} \sqrt{\frac{6}{\pi}}e^{-6 x^2} dx = \int_{-\infty}^{-a} e^{-\pi \left(\sqrt{\frac{6}{\pi}} x\right)^2} d\left(\sqrt{\frac{6}{\pi}} x\right) \\ = \int_{-\infty}^{-\sqrt{\frac{6}{\pi}} a} e^{-\pi x^2} dx = \Phi\left(-\sqrt{\frac{6}{\pi}} a\right) > \Phi\left(\sqrt{\frac{6}{\pi}} \sqrt{\frac{\pi}{6}} \Phi^{-1}(\varepsilon) \right) = \varepsilon
		\end{multline*}
		Thus for large enough $n$ a part of volume at least $\varepsilon$ is going to be cut off. By symmetry the same is true for an orthogonal hyperplane section of our cube corresponding to $[\frac{1}{2}\sqrt{n} + a; \sqrt{n}]$ on the main diagonal.

		So for $n$ large enough we get two subsets of volume at least $\varepsilon$ at a distance $2a$. But $a$ was chosen as an arbitrary number smaller than
		\[
			-\sqrt{\frac{\pi}{6}} \Phi^{-1}(\varepsilon),
		\]
		from which the statement of the theorem follows.
	\end{proof}

	\begin{theorem}
		\label{simplex}
		When $K_n$ are the unit-volume simplexes we have
		\[
			\liminf_{n \to \infty} d_n(\varepsilon) \geq -\frac{\sqrt{2}}{e} \ln(2\varepsilon).
		\]
	\end{theorem}
	\begin{proof}
		The volume of the unit simplex $\Delta_n$ defined by
		\[
			\Delta_n = \{ (x_1, \ldots, x_n) \in \mathbb{R}_+^n \mid x_1 + \ldots + x_n = 1 \}
		\]
		is equal to
		\[
			\frac{n \sqrt{n}}{n!}
		\]
		
		If we set
		\[
			\omega_n = \left(\frac{n!}{n\sqrt{n}}\right)^{\frac{1}{n - 1}} \sim \frac{n}{e},
		\]
		then $\omega_n \Delta_n$ will be a regular unit-volume simplex, whose side length is
		\[
			\sqrt{2} \omega_n
		\]

		Now consider any number $\varepsilon^{\prime}$ such that
		\[
			\varepsilon < \varepsilon^{\prime} < \frac{1}{2}
		\]
		Let $P$ and $Q$ be two vertices of $\omega_n \Delta_n$. Hyperplane passing through the midpoint of the side $PQ$ and orthogonal to it divides our simplex into two parts of equal volume. We set
		\[
			\alpha = (2\varepsilon^{\prime})^{\frac{1}{n - 1}}.
		\]
		Consider the image of the part containing $P$ after a homotethy with center at $P$ and coefficient $\alpha$, the resulting subset of our simplex will be of volume $\frac{1}{2} \alpha^{n - 1} = \varepsilon^{\prime} > \varepsilon$. Analogously, we construct a subset of volume $\varepsilon^{\prime}$ corresponding to the vertex $Q$.

		Note that we have two homotethies with coefficient $\alpha$ applied to the halves of $PQ$. The distance between our subsets will be equal to
		\[
			\sqrt{2} \omega_n \left(1 - \alpha\right) = \sqrt{2} \omega_n \left(1 - (2\varepsilon^{\prime})^{\frac{1}{n - 1}}\right)
		\]
		Now we take limit
		\begin{multline*}
			\lim_{n \to \infty} \sqrt{2} \omega_n \left(1 - (2\varepsilon^{\prime})^{\frac{1}{n - 1}}\right) = \lim_{n \to \infty} \frac{\sqrt{2}}{e} n \left(1 - (2\varepsilon^{\prime})^{\frac{1}{n - 1}}\right) \\ = \lim_{n \to \infty} -\frac{\sqrt{2}}{e} \frac{n}{n - 1} \frac{(2\varepsilon^{\prime})^0 - (2\varepsilon^{\prime})^\frac{1}{n - 1}}{0 - \frac{1}{n - 1}} = -\frac{\sqrt{2}}{e} \frac{d}{dt} (2\varepsilon^{\prime})^t\Big|_{t = 0} = -\frac{\sqrt{2}}{e} \ln(2\varepsilon^{\prime})
		\end{multline*}

		In other words, for each $n$ we have two subsets of $\omega_n \Delta_n$ of volume at least $\varepsilon$, and the distance between them tends to
		\[
			-\frac{\sqrt{2}}{e} \ln(2\varepsilon^{\prime})
		\]
		as $n$ goes to infinity. But $\varepsilon^{\prime}$ was chosen as an arbitrary number from $(\varepsilon; \frac{1}{2})$, from which the statement of the theorem follows.
	\end{proof}

	\begin{theorem}
		\label{l_p_balls}
		When $K_n$ are the unit-volume $\ell_p$ balls for $p \in [1; 2]$ we have
		\[
			\liminf_{n \to \infty} d_n(\varepsilon) \geq -2\Psi_p^{-1}(\varepsilon),
		\]
		where function $-2\Psi_p^{-1}(\varepsilon)$(see Appendix \ref{v_n_s_n}) is asymptotically equivalent to
		\[
			\frac{1}{e^{\frac{1}{p}} \Gamma\left(1 + \frac{1}{p}\right)} (-\ln \varepsilon)^{\frac{1}{p}}
		\]
		as $\varepsilon \to 0$.
	\end{theorem}
	\begin{proof}
		In the unit-volume $\ell_p$ ball $\omega_n \ell_p^n$, where
		\[
			\omega_n = \frac{\Gamma\left(1 + \frac{n}{p}\right)^{\frac{1}{n}}}{2\Gamma\left(1 + \frac{1}{p}\right)} \sim \frac{n^{\frac{1}{p}}}{2\Gamma\left(1 + \frac{1}{p}\right)(pe)^{\frac{1}{p}}},
		\]
		consider the segment from $(-\omega_n, 0, \ldots, 0)$ to $(\omega_n, 0, \ldots, 0)$. We will be considering the hyperplanes orthogonal to this segment, i.\,e.\,hyperplanes defined by $X_1 = t$. By $V_n(t)$ denote the function that measures the volume of the part given by $X_1 \geq t$.

		We use notation from Appendix \ref{v_n_s_n}, where it was established that functions $V_n(x)$ uniformly converge to $\Psi_p(-x)$ on $(-\infty; +\infty)$. Pick an arbitrary number $a$ such that
		\[
			\Psi_p^{-1}(\varepsilon) < -a < 0
		\]
		Then for all sufficiently large $n$ we shall have
		\[
			V_n(a) > \varepsilon
		\]
		Because of symmetry, for all sufficiently large $n$ we would have two bodies in $\omega_n \ell_p^n$ of volume at least $\varepsilon$ at a distance at least $2a$. And, since the choice of $a$ above was arbitrary, we would have
		\[
			\liminf_{n \to \infty} d_n(\varepsilon) \geq -2\Psi_p^{-1}(\varepsilon)
		\]
	\end{proof}

	Assume that we have a family $K_n$ of bounded unit-volume centrally symmetric bodies that are not necessarily convex, where the origin will be the center of symmetry for each $K_n$. By $\mu_n$ denote the uniform probability measure on $K_n$.

	Fix $\varepsilon \in (0; \frac{1}{2})$. Since $K_n$ is bounded, it has a finite diameter, thus for each $K_n$ we could consider $d_n(\varepsilon)$ -- the supremum of all possible distances between two subsets of $K_n$, whose volume is $\varepsilon$.

	Turns out that a simple averaging argument gives us.
	\begin{theorem}
		\label{general_sup}
		For every $\varepsilon \in (0; \frac{1}{2})$
		\[
			\liminf_{n \to \infty} d_n(\varepsilon) \geq -2\frac{1}{\sqrt{e}} \Phi^{-1}(\varepsilon),
		\]
		where function $-2\frac{1}{\sqrt{e}} \Phi^{-1}(\varepsilon)$ is asymptotically equivalent to
		\[
			-2\frac{1}{\sqrt{\pi e}} \sqrt{-\ln \varepsilon}
		\]
		as $\varepsilon \to 0$.
	\end{theorem}
	\begin{proof}
		Consider a number $d$ such that
		\[
			\frac{1}{\sqrt{e}}\Phi^{-1}(\varepsilon) < -d < 0
		\]
		We would like to show that for all sufficiently large $n$ there would be a direction defined by a unit vector $u$ on the sphere $S^{n - 1}$ such that the part of our body $K_n$ where $\langle x, u \rangle \leq -d$ would be of volume at least $\varepsilon$. By central symmetry the part where $\langle x, u \rangle \geq d$ would be of the same volume. But then we will be having two subsets of $K_n$ of volume at least $\varepsilon$ at a distance at least $2d$, where $d$ was chosen as an arbitrary number lesser than
		\[
			-\frac{1}{\sqrt{e}} \Phi^{-1}(\varepsilon),
		\]
		from which the statement of the theorem would follow.

		We are looking for a direction $u \in S^{n - 1}$ with
		\[
			\mu_n(\langle x, u \rangle \leq -d) \geq \varepsilon
		\]
		If the average value of $\mu_n(\langle x, u \rangle \leq -d)$ is greater than $\varepsilon$, then such a direction will exist.

		By $\vartheta_n$ denote the uniform probability measure on $S^{n - 1}$. We would like to have
		\begin{equation}
			\label{step_1}
			\int_{S^{n - 1}} \mu_n(\langle x, u \rangle \leq -d) d\vartheta_n \geq \varepsilon
		\end{equation}
		Volume of the hyperplane cut defined by $\langle x, u \rangle \leq -d$ is the integral of the indicator function $[\langle x, u \rangle \leq -d]$ on $K_n$, which allows to rewrite (\ref{step_1}) as
		\[
			\int_{S^{n - 1}} \int_{K_n} [\langle x, u \rangle \leq -d] d\mu_n d\vartheta_n \geq \varepsilon
		\]
		We may switch the order of integration
		\begin{equation}
			\label{step_2}
			\int_{K_n} \left(\int_{S^{n - 1}} [\langle x, u \rangle \leq -d] d\vartheta_n \right) d\mu_n \geq \varepsilon
		\end{equation}
		We would like to consider the integrand of (\ref{step_2}) for $x \neq 0$. First, we rewrite $[\langle x, u \rangle \leq -d]$ as
		\[
			\left[\sqrt{n} \left\langle \frac{x}{\Vert x \Vert_2}, u \right\rangle \leq - \frac{\sqrt{n} d}{\Vert x \Vert_2}\right]
		\]
		Now we could think of $u$ as of a random vector on the unit sphere $S^{n - 1}$. Then
		\[
			\left\langle \frac{x}{\Vert x \Vert_2}, u \right\rangle
		\]
		corresponds to the projection of $u$ on the diameter from $-\frac{x}{\Vert x \Vert_2}$ to $+\frac{x}{\Vert x \Vert_2}$. And thus by theorem 1 from \cite{ball_section} as $n$ goes to infinity the distribution of random variable
		\[
			\sqrt{n} \left\langle \frac{x}{\Vert x \Vert_2}, u \right\rangle
		\]
		converges in total variation to the standard normal distribution.

		Note that the value of
		\[
			\int_{S^{n - 1}} \left[\sqrt{n} \left\langle \frac{x}{\Vert x \Vert_2}, u \right\rangle \leq - \frac{\sqrt{n} d}{\Vert x \Vert_2}\right] d\vartheta_n
		\]
		only depends on $- \frac{\sqrt{n} d}{\Vert x \Vert_2}$. So we define the function
		\[
			\Psi_n(x) = \int_{S^{n - 1}} \left[\sqrt{n} \left\langle \frac{x}{\Vert x \Vert_2}, u \right\rangle \leq x\right] d\vartheta_n
		\]
		As was remarked above, for every $x$
		\[
			\lim_{n \to \infty} \Psi_n(x) = \frac{1}{\sqrt{2\pi}}\int_{-\infty}^{x} e^{-\frac{1}{2}x^2} dx
		\]
		Since functions $\Psi_n$ are monotonic, we should also have
		\[
			\lim_{n \to \infty} \Psi_n(x_n) = \frac{1}{\sqrt{2\pi}}\int_{-\infty}^{x} e^{-\frac{1}{2}x^2} dx
		\]
		for any sequence $x_n$ that tends to $x$ as $n \to \infty$.

		We rewrite (\ref{step_2}) as
		\begin{equation}
			\label{step_3}
			\int_{K_n \setminus \{ 0 \}} \Psi_n\left(-\frac{\sqrt{n} d}{\Vert x \Vert_2}\right) d\mu_n \geq \varepsilon
		\end{equation}
		The function $\Psi_n$ is non-decreasing, thus
		\[
			\Vert x \Vert_2 \geq r > 0 \Rightarrow \Psi_n\left(-\frac{\sqrt{n} d}{\Vert x \Vert_2}\right) \geq \Psi_n\left(-\frac{\sqrt{n} d}{r}\right)
		\]
		So for (\ref{step_3}) condition
		\begin{equation}
			\label{step_4}
			\int_{K_n \setminus B^n_r} \Psi_n\left(-\frac{\sqrt{n} d}{r}\right) d\mu_n \geq \varepsilon
		\end{equation}
		would be sufficient, we may go further and require
		\begin{equation}
			\label{step_5}	
			(1 - V(B^n_r)) \Psi_n\left(-\frac{\sqrt{n} d}{r}\right) \geq \varepsilon,
		\end{equation}
		where $V(B^n_r)$ is the volume of the $n$-ball with radius $r$.

		The volume of the unit $n$-ball $B^n_1$ is
		\[
			\frac{\sqrt{\pi}^n}{\Gamma\left(\frac{n}{2} + 1\right)}.
		\]
		The unit volume corresponds to the radius
		\[
			\omega_n = \frac{\Gamma\left(\frac{n}{2} + 1\right)^{\frac{1}{n}}}{\sqrt{\pi}} \sim \sqrt{\frac{n}{2\pi e}}
		\]
		Pick a number $\alpha \in (0; 1)$ such that
		\begin{equation}
			\label{alpha_cond}
			\alpha \frac{1}{\sqrt{e}} \Phi^{-1}(\varepsilon) < -d
		\end{equation}
		and set
		\[
			r = \alpha \omega_n
		\]

		Inequality (\ref{step_5}) turns into
		\begin{equation}
			\label{step_6}
			(1 - \alpha^n) \Psi_n\left(-\frac{\sqrt{n} d}{\alpha \omega_n}\right) \geq \varepsilon
		\end{equation}
		Note that
		\[
			\lim_{n \to \infty} \frac{\sqrt{n} d}{\alpha \omega_n} = \lim_{n \to \infty} \sqrt{\frac{2\pi e}{n}} \frac{\sqrt{n} d}{\alpha} = \frac{\sqrt{2\pi e} d}{\alpha}
		\]
		Now take limit of the left side of (\ref{step_6}) and apply condition (\ref{alpha_cond})
		\begin{multline*}
			\lim_{n \to \infty} (1 - \alpha^n) \Psi_n\left(-\frac{\sqrt{n} d}{\alpha \omega_n}\right) = \frac{1}{\sqrt{2\pi}} \int_{-\infty}^{-\frac{\sqrt{2\pi e} d}{\alpha}} e^{-\frac{1}{2}x^2} dx \\ = \int_{-\infty}^{-\frac{\sqrt{2\pi e} d}{\alpha}} e^{-\pi \left(\frac{1}{\sqrt{2\pi}} x\right)^2} d\left(\frac{1}{\sqrt{2\pi}} x\right) = \int_{-\infty}^{-\frac{\sqrt{e} d}{\alpha}} e^{-\pi x^2} dx = \Phi\left(-\frac{\sqrt{e}d}{\alpha}\right) \\ > \Phi(\Phi^{-1}(\varepsilon)) = \varepsilon
		\end{multline*}

		Thus for all sufficiently large $n$
		\[
			(1 - \alpha^n) \Psi_n\left(-\frac{\sqrt{n} d}{\alpha \omega_n}\right) > \varepsilon
		\]
		and, consequently,
		\begin{equation}
			\int_{S^{n - 1}} \mu_n(\langle x, u \rangle \leq -d) d\vartheta_n \geq \varepsilon,
		\end{equation}
		which means that for all sufficiently large $n$ the desired direction $u \in S^{n - 1}$ with property
		\[
			\mu_n(\langle x, u \rangle \leq -d) \geq \varepsilon
		\]
		exists.
	\end{proof}

	Theorem \ref{ball_sup} is an immediate corollary of the Theorem \ref{general_sup}.

	\medskip

	\textbf{Conclusions.} Roughly speaking, the results established in \cite{K07}, \cite{KLARTAG2007284} tell us that generally hyperplane sections of convex bodies across arbitrary directions lead to a gaussian distribution, suggesting that the asymptotic behavior of $d_n(\varepsilon)$ different from $\Phi^{-1}(\varepsilon)$ is probably caused by a few degenerate directions. For example, Theorems \ref{simplex} and \ref{simplex_upper} tell us that the asymptotic behavior of $d_n(\varepsilon)$ for simplexes corresponds to a function $-\ln \varepsilon$, but one could also note that the simplex $\Delta_n$ is unusually <<stretched>> in $n$ directions corresponding to its corners, and that almost all of its volume by Theorem \ref{concetration} is concentrated in the euclidean ball of a diameter much smaller than the diameter of $\Delta_n$.

	\newpage

	\section{Discrete isoperimetric problem.}

	\textbf{Introduction.} In this section instead of considering the euclidean distance between two subsets of volume $\varepsilon$ we will be considering the Manhattan distance. We can no longer say that for fixed $\varepsilon$ this distance is bounded, but we will derive a result(Theorem \ref{discrete_cube}) concerning its asymptotic behaviour in the case of the unit cube $[0; 1]^n$. The key idea in the proof of it would be to replace the unit cube $[0; 1]^n$ with a lattice, for which the solution of our problem is already known(see Theorem \ref{discrete_distance}).

	We consider the Manhattan distance $d$ in the unit cube $[0; 1]^n$. We may ask a similar question: what is the largest Manhattan distance between two bodies of volume $\varepsilon > 0$ in the unit cube $[0; 1]^n$?

	Turns out this problem could be dealt with by discretization. Consider a lattice $L = \{\frac{0}{m}, \frac{1}{m}, \ldots, \frac{m}{m}\}^n$ inside $[0; 1]^n$, two points $x$ and $y$ in $L$ are adjacent whenever $d(x, y) = \frac{1}{m}$. By $t$-boundary $A_{(t)}$ of a subset $A \subseteq L$ we mean the set of all points of $L$ that are at a distance at most $\frac{t}{m}$ from $A$. The latter concept is analagous to $\delta$-enlargements, and one could think of $|A_{(1)} \setminus A|$ as an analogue of the surface area. This gives rise to the discrete isoperimetirc problem: how large the $t$-boundary of a set $A \subseteq L$ with fixed size can be? This question was answered in \cite{BOLLOBAS199147}

	\begin{theorem}[{{\cite[Corollary 9]{BOLLOBAS199147}}}]
		Let $A \subset [k]^n$. For any $t = 0, 1, \ldots,$ the $t$-boundary of $A$ is at least as large as the $t$-boundary of the first $|A|$ elements of $[k]^n$ in the simplicial order.
	\end{theorem}

	Here $[k]^n$ is the lattice $\{0, \ldots, k - 1\}^n$. Simplicial order on $[k]^n$ is defined by setting $x < y$ if either $\sum x_i < \sum y_i$, or $\sum x_i = \sum y_i$ and for some $j$ we have $x_j > y_j$ and $x_i = y_i$ for all $i < j$. This discrete isoperimetric inequality leads to

	\begin{theorem}[{{\cite[Corollary 10]{BOLLOBAS199147}}}]
		\label{discrete_distance}
		There are sets $A, B \subset [k]^n$ with $|A| = r, |B| = s$, and $d(A, B) \geq d$ iff the distance between the first $r$ and the last $s$ elements of the simplicial order on $[k]^n$ is at least $d$.
	\end{theorem}

	From this discrete version of the problem considered in the first pararaph we can derive

	\begin{theorem}
		\label{discrete_cube}
		If by $d_n(\varepsilon)$ we denote the largest Manhattan distance between two bodies of volume $\varepsilon \in (0; \frac{1}{2})$ in the unit cube $[0; 1]^n$, then
		\[
			\lim_{n \to \infty} \frac{d_n(\varepsilon)}{\sqrt{n}} = -2\sqrt{\frac{\pi}{6}}\Phi^{-1}(\varepsilon)
		\]
		The function $-2\sqrt{\frac{\pi}{6}} \Phi^{-1}(\varepsilon)$ is asymptotically equivalent to
		\[
			\frac{2}{\sqrt{6}}\sqrt{- \ln \varepsilon}
		\]
		as $\varepsilon \to 0$. 
	\end{theorem}
	\begin{proof}
		Pick a number $a$ such that
		\[
			-a < \sqrt{\frac{\pi}{6}} \Phi^{-1}(\varepsilon) < 0
		\]
		If by $V_n$ and $W_n$ we denote the subsets of $[0; 1]^n$ defined by inequalities $\sum x_i \leq \frac{1}{2}\sqrt{n} - a$ and $\sum x_i \geq \frac{1}{2}\sqrt{n} + a$, respectively, then by applying the central limit theorem just as we did in Theorem \ref{cubes} we shall get
		\begin{multline*}
			\lim_{n \to \infty} \mu(V_n) = \lim_{n \to \infty} \mu(W_n) = \int_{-\infty}^{-a} \sqrt{\frac{6}{\pi}} e^{-6 x^2} dx = \Phi\left( -\sqrt{\frac{6}{\pi}} a \right) \\ < \Phi\left( \sqrt{\frac{6}{\pi}} \sqrt{\frac{\pi}{6}} \Phi^{-1}(\varepsilon) \right) = \varepsilon
		\end{multline*}
		So there is a number $N$ such that for all $n > N$ the volumes of $V_n$ and $W_n$ are going to be lesser than $\varepsilon$, and it could be noted that the Manhattan distance between them is equal to $2a\sqrt{n}$. From now on we assume that $n > N$.

		Now consider two bodies $A$ and $B$ inside $[0; 1]^n$ of volume $\varepsilon$ and a lattice $L_m = \{\frac{0}{m}, \ldots, \frac{m}{m}\}^n$. The discretization is justified by the fact that the distance does not decrease when we restrict our attention to the lattice $L_m$
		\[
			d(A, B) \leq d(A \cap L_m, B \cap L_m)
		\]
		Since $\mu(V_n) < \mu(A)$ and $\mu(W_n) < \mu(B)$, for all $m$ large enough we shall have
		\[
			|V_n \cap L_m| \leq |A \cap L_m| \quad |W_m \cap L_m| \leq |B \cap L_m|
		\]
		And by Theorem \ref{discrete_distance} this implies
		\[
			d(A \cap L_m, B \cap L_m) \leq d(V_n \cap L_m, W_n \cap L_m),
		\]
		because $V_n \cap L_m$ and $W_n \cap L_m$ are the first $|V_n \cap L_m|$ and the last $|W_n \cap L_m|$ elements of the simplicial order on $L_m$, respectively.

		We conclude that for all $m$ large enough
		\[
			d(A, B) \leq d(V_n \cap L_m, W_n \cap L_m)
		\]
		Also
		\[
			\lim_{m \to \infty} d(V_n \cap L_m, W_n \cap L_m) = d(V_n, W_n) = 2a\sqrt{n}
		\]
		Thus for all $n > N$
		\[
			d(A, B) \leq 2a\sqrt{n},
		\]
		but $a$ was chosen here as an arbitrary number greater than
		$
			-\sqrt{\frac{\pi}{6}} \Phi^{-1}(\varepsilon),
		$
		from which
		\begin{equation}
			\label{_one}
			\limsup_{n \to \infty} \frac{d_n(\varepsilon)}{\sqrt{n}} \leq -2\sqrt{\frac{\pi}{6}}\Phi^{-1}(\varepsilon)
		\end{equation}
		follows.

		If we assume that
		\[
			\sqrt{\frac{\pi}{6}} \Phi^{-1}(\varepsilon) < -a < 0,
		\]
		then as it was already shown in the proof of Theorem \ref{cubes} for all $n$ large enough
		\[
			\mu(V_n) = \mu(W_n) > \varepsilon
		\]
		But the Manhattan distance between $V_n$ and $W_n$ is $2a\sqrt{n}$.

		From this observation we can conclude
		\[
			\liminf_{n \to \infty} \frac{d_n(\varepsilon)}{\sqrt{n}} \geq -2\sqrt{\frac{\pi}{6}}\Phi^{-1}(\varepsilon),
		\]
		which together with inequality (\ref{_one}) gives us the statement of the theorem.
	\end{proof}

	\textbf{Conclusions.} There are also other variations of the discrete isoperimetric problem. For example, Theorem \ref{discrete_distance} has an analogue for the non-negative orthant of the integer lattice $\mathbb{Z}^n_+$(see \cite{10.2307/2100602} and \cite[Theorem 4]{BOLLOBAS199147}).

	\newpage	

	\section{Conclusions.}

	Upper bounds on the distance between subsets in unit-volume $\ell_p$ balls with $p \in [1;2] \cup \{+\infty\}$ have been established. A case of $p \in (2; +\infty)$ remains. Since asymptotically our estimates for unit-volume euclidean balls($p = 2$) and for unit cubes($p = +\infty$) are the same, we expect to have a similar asymptotic behavior for all $p \in [2; +\infty]$. As was remarked before both cases of $p = 2$ and $p = +\infty$ could be approached by providing a Lipschitz map that transforms Gaussian measure into the uniform measure on the corresponding $\ell_p$ balls. Perhaps, the same approach might work out for $p \in (2; +\infty)$. On the page $4$ of \cite{B07} it was remarked that the uniform measure on $\ell_p^n$ balls with $p \in [2; +\infty]$ <<can be obtained from the canonical Gaussian measure as Lipschitz transform>>. Although, we are not sure about the exact meaning and implications of this statement.

	In Theorem \ref{general_sup} we established a sort of a general lower bound regarding our problem. But is it possible to find a general upper bound on the distance between two subsets of volume $\varepsilon \in (0; \frac{1}{2})$ in a unit-volume convex body? Clearly, we are going to have to consider only some certain <<good>> convex bodies. For example, in a convex body stretched far in a particular direction two subsets of fixed volume could be at an arbitrarily large distance. One such notion of a <<good>> convex body is related to the isotropic position. But this means that we are looking for general estimates on the isoperimetric problem in isotropic convex bodies, this appears to be a complicated open question(Kannan-Lov\'{a}sz-Simonovits conjecture, see \cite{KLS}). Furthermore, even the relation between the volume and the isotropic constant of a convex body seems to be at the core of another open problem(isotropic constant conjecture, see \cite{B91}, \cite{S95}). In paper \cite{KLS_U} a very good lower bound related to the KLS conjecture was proven. There is a lot of material on isotropic convex bodies, for example, \cite{isotropy}.

	Log-concave probability measures generalize uniform measures on convex bodies. So one might try to consider the problem in a more general setting. Gaussian measures are log-concave, which means that in a more general setting our estimates may be sharp for some specific distributions in the one-dimensional case. In paper \cite{MS} some general estimates on the specific version of the isoperimetric problem were proven(see, for example, theorem $1.3$).

	\nocite{*}
	\printbibliography

	\newpage

	\appendix

	\section{Asymptotic behavior of $\Phi^{-1}$.}
	\label{asymp_phi_inv}

	We want to show that
	\begin{equation}
		\label{asymp_limit}
		\lim_{\varepsilon \to 0} \frac{\frac{1}{\sqrt{\pi}} \sqrt{-\ln \varepsilon}}{-\Phi^{-1}(\varepsilon)} = 1
	\end{equation}
	We begin with the following observation
	\[
			\frac{d}{dx}\left(\frac{e^{-\pi x^2}}{-2\pi x}\right) = e^{-\pi x^2} + \frac{e^{-\pi x^2}}{2 \pi x^2}
	\]
	By integrating both parts from $-\infty$ to $a < 0$ we get
	\begin{equation}
		\label{phi_integral}
		\frac{e^{-\pi a^2}}{-2\pi a} = \left.\left(\frac{e^{-\pi x^2}}{-2\pi x}\right)\right|_{-\infty}^{a} = \int_{-\infty}^{a} e^{-\pi x^2}\left(1 + \frac{1}{2\pi x^2}\right)dx
	\end{equation}

	From this we can derive the following upper bound on $\Phi$
	\[
		\Phi(a) \leq \int_{-\infty}^{a} e^{-\pi x^2}\left(1 + \frac{1}{2\pi x^2}\right)dx = \frac{e^{-\pi a^2}}{-2\pi a}
	\]
	Now we apply $-\Phi^{-1}$ to both sides
	\[
		-a \geq -\Phi^{-1}\left( \frac{e^{-\pi a^2}}{-2\pi a} \right)
	\]
	Keep in mind that $a < 0$. By putting $\frac{e^{-\pi a^2}}{-2\pi a}$ instead of $\varepsilon$ into the expression from (\ref{asymp_limit}) we arrive at
	\[
		\frac{\frac{1}{\sqrt{\pi}}\sqrt{\pi a^2 + \ln(-2\pi a)}}{-\Phi^{-1}(\frac{e^{-\pi a^2}}{-2\pi a})} \geq \frac{\frac{1}{\sqrt{\pi}}\sqrt{\pi a^2 + \ln(-2\pi a)}}{-a}
	\]
	But as $a$ goes to $-\infty$
	\[
		\frac{e^{-\pi a^2}}{-2\pi a} \to 0 \textrm{ and } \frac{\frac{1}{\sqrt{\pi}}\sqrt{\pi a^2 + \ln(-2\pi a)}}{-a} \to 1
	\]
	So
	\begin{equation}
		\label{limit_inf}
		\liminf_{\varepsilon \to 0} \frac{\frac{1}{\sqrt{\pi}} \sqrt{-\ln \varepsilon}}{-\Phi^{-1}(\varepsilon)} \geq 1
	\end{equation}
	if defined.

	From (\ref{phi_integral}) we can also derive a lower bound on $\Phi$
	\[
		\Phi(a) \geq \left(1 + \frac{1}{2\pi a^2} \right)^{-1} \int_{-\infty}^{a} e^{-\pi x^2}\left(1 + \frac{1}{2\pi x^2}\right)dx = \frac{e^{-\pi a^2}}{-2\pi a} \left(1 + \frac{1}{2\pi a^2} \right)^{-1}
	\]
	Again we apply $-\Phi^{-1}$ to both sides and get
	\[
		-a \leq -\Phi^{-1}\left( \frac{e^{-\pi a^2}}{-2\pi a} \left(1 + \frac{1}{2\pi a^2} \right)^{-1} \right)
	\]
	By putting $\frac{e^{-\pi a^2}}{-2\pi a} \left(1 + \frac{1}{2\pi a^2} \right)^{-1}$ instead of $\varepsilon$ into the expression from (\ref{asymp_limit}) we arrive at
	\[
		\frac{\frac{1}{\sqrt{\pi}}\sqrt{\pi a^2 + \ln(-2\pi a) + \ln\left( 1 + \frac{1}{2\pi a^2} \right)}}{-\Phi^{-1}\left( \frac{e^{-\pi a^2}}{-2\pi a} \left(1 + \frac{1}{2\pi a^2} \right)^{-1} \right)} \leq \frac{\frac{1}{\sqrt{\pi}}\sqrt{\pi a^2 + \ln(-2\pi a) + \ln\left( 1 + \frac{1}{2\pi a^2} \right)}}{-a}
	\]
	But as $a$ goes to $-\infty$
	\[
		\frac{e^{-\pi a^2}}{-2\pi a} \left(1 + \frac{1}{2\pi a^2} \right)^{-1} \to 0,
	\]
	\[
		\frac{\frac{1}{\sqrt{\pi}}\sqrt{\pi a^2 + \ln(-2\pi a) + \ln\left( 1 + \frac{1}{2\pi a^2} \right)}}{-a} \to 1
	\]
	Thus
	\begin{equation}
		\label{limit_sup}
		\limsup_{\varepsilon \to 0} \frac{\frac{1}{\sqrt{\pi}} \sqrt{-\ln \varepsilon}}{-\Phi^{-1}(\varepsilon)} \leq 1
	\end{equation}
	Together (\ref{limit_inf}) and (\ref{limit_sup}) give
	\[
		\lim_{\varepsilon \to 0} \frac{\frac{1}{\sqrt{\pi}} \sqrt{-\ln \varepsilon}}{-\Phi^{-1}(\varepsilon)} = 1
	\]

	\section{Function $x (-\log x)^{1 - \frac{1}{p}}$ is increasing on $(0; \frac{1}{2}]$.}
	\label{weird_func_inc}

	To show that function $x (-\log x)^{1 - \frac{1}{p}}$ is increasing on $(0; \frac{1}{2}]$ for $p \in [1; 2]$ we simply take the derivative
	\begin{multline*}
		\frac{d}{dx}\left( x (-\log x)^{1 - \frac{1}{p}} \right) = (-\log x)^{1 - \frac{1}{p}} + x \frac{d}{dx} (-\log x)^{1 - \frac{1}{p}} \\ = (-\log x)^{1 - \frac{1}{p}} + x \frac{d (-\log x)^{1 - \frac{1}{p}}}{d(-\log x)}\frac{d(-\log x)}{dx} \\ = (-\log x)^{1 - \frac{1}{p}} + \left(1 - \frac{1}{p}\right)x (-\log x)^{- \frac{1}{p}} \left(-\frac{1}{x}\right) \\ = (-\log x)^{1 - \frac{1}{p}} - \left(1 - \frac{1}{p}\right) (-\log x)^{- \frac{1}{p}} \\ = (-\log x)^{- \frac{1}{p}} \left( (-\log x) - \left(1 - \frac{1}{p}\right) \right)
	\end{multline*}
	Clearly, $(-\log x)^{- \frac{1}{p}} > 0$. On the half-interval $(0; \frac{1}{2}]$ we should have
	\[
		(-\log x) - \left(1 - \frac{1}{p}\right) \geq \left(-\log \frac{1}{2}\right) - \frac{1}{2} = \log 2 - \frac{1}{2} > 0,
	\]
	which means that for all $x \in (0; \frac{1}{2}]$
	\[
		\frac{d}{dx}\left( x (-\log x)^{1 - \frac{1}{p}} \right) > 0
	\]

	\section{Functions $V_n(x)$ and $S_n(x)$ in limit.}
	\label{v_n_s_n}

	For $p \in [1; \infty)$ consider the unit $\ell_p^n$ ball defined by
	\[
		|x_1|^p + \ldots + |x_n|^p = 1
	\]
	It has volume
	\[
		\frac{\left(2 \Gamma\left(1 + \frac{1}{p}\right)\right)^n}{\Gamma\left(1 + \frac{n}{p}\right)}
	\]
	Let
	\begin{equation}
		\omega_n = \frac{\Gamma\left(1 + \frac{n}{p}\right)^{\frac{1}{n}}}{2\Gamma\left(1 + \frac{1}{p}\right)} \sim \frac{n^{\frac{1}{p}}}{2\Gamma\left(1 + \frac{1}{p}\right)(pe)^{\frac{1}{p}}}
	\end{equation}

	For a non-negative number $x$ by $S_n(x)$ denote the volume of the section $x_1 = x$ of the unit-volume body $\omega_n \ell_p^n$. For $x > \omega_n$ we assume $S_n(x) = 0$. For $x \leq \omega_n$ the $\ell_p^{n - 1}$ ball corresponding to this section is defined by equation
	\[
		|x_2|^p + \ldots + |x_n|^p = \omega_n^p - x^p
	\]
	It has volume
	\begin{multline*}
		S_n(x) = (\omega_n^p - x^p)^{\frac{n - 1}{p}} \frac{\left(2 \Gamma\left(1 + \frac{1}{p}\right)\right)^{n - 1}}{\Gamma\left(1 + \frac{n - 1}{p}\right)} \\ = \left(1 - \frac{x^p}{\omega_n^p} \right)^{\frac{n - 1}{p}} \frac{1}{\omega_n} \omega_n^n \frac{\left(2 \Gamma\left(1 + \frac{1}{p}\right)\right)^{n - 1}}{\Gamma\left(1 + \frac{n - 1}{p}\right)} \\ = \left(1 - \frac{x^p}{\omega_n^p} \right)^{\frac{n - 1}{p}} \frac{1}{\omega_n} \frac{\Gamma\left(1 + \frac{n}{p}\right)}{\left(2 \Gamma\left(1 + \frac{1}{p}\right)\right)^n} \frac{\left(2 \Gamma\left(1 + \frac{1}{p}\right)\right)^{n - 1}}{\Gamma\left(1 + \frac{n - 1}{p}\right)} \\ =
		\left(1 - \frac{x^p}{\omega_n^p} \right)^{\frac{n - 1}{p}} \frac{\Gamma\left(1 + \frac{n}{p}\right)}{2 \omega_n \Gamma\left(1 + \frac{1}{p}\right) \Gamma\left(1 + \frac{n - 1}{p}\right)}
	\end{multline*}

	Since
	\[
		\lim_{x \to \infty} \frac{\Gamma(x + \alpha)}{\Gamma(x) x^{\alpha}} = 1,
	\]
	we shall have
	\begin{multline*}
		\lim_{n \to \infty}  \frac{\Gamma\left(1 + \frac{n}{p}\right)}{2 \omega_n \Gamma\left(1 + \frac{1}{p}\right) \Gamma\left(1 + \frac{n - 1}{p}\right)} = \lim_{n \to \infty} \left(\omega_n 2 \Gamma\left(1 + \frac{1}{p}\right)\right)^{-1} \frac{\Gamma\left(1 + \frac{n}{p}\right)}{\Gamma\left(1 + \frac{n - 1}{p}\right)} \\ = \lim_{n \to \infty} \left( \frac{n^{\frac{1}{p}}}{2\Gamma\left(1 + \frac{1}{p}\right)(pe)^{\frac{1}{p}}} 2 \Gamma\left(1 + \frac{1}{p}\right)\right)^{-1} \left(\frac{n}{p} \right)^{\frac{1}{p}} = e^{\frac{1}{p}}
	\end{multline*}

	Also
	\begin{equation*}
		\lim_{n \to \infty} \frac{n}{p \omega_n^p} = \lim_{n \to \infty} n\left(p \frac{n}{2^p \Gamma\left(1 + \frac{1}{p}\right)^{p}pe}\right)^{-1} = 2^p \Gamma\left(1 + \frac{1}{p}\right)^p e
	\end{equation*}
	Thus on the segment $[0; D]$ the sequence of functions
	\[
		\left(1 - \frac{x^p}{\omega_n^p} \right)^{\frac{n - 1}{p}} = \left(1 - \frac{x^p}{\omega_n^p} \right) \left(1 - \frac{x^p \frac{n}{p \omega_n^p}}{\frac{n}{p}} \right)^{\frac{n}{p}}
	\]
	uniformly converges to
	\[
		e^{-x^p 2^p \Gamma\left(1 + \frac{1}{p}\right)^p e} = e^{-\left(2 \Gamma\left( 1 + \frac{1}{p}\right) e^{\frac{1}{p}} x\right)^p}
	\]
	as $n$ tends to $\infty$.

	Thus on the segment $[0; D]$ functions $S_n(x)$ uniformly converge to
	\[
		\psi_p(x) = e^{\frac{1}{p}} e^{-x^p 2^p \Gamma\left(1 + \frac{1}{p}\right)^p e} = e^{\frac{1}{p}} e^{-\left(2 \Gamma\left( 1 + \frac{1}{p}\right) e^{\frac{1}{p}} x\right)^p}
	\]
	For an arbitrary $\varepsilon > 0$ pick a number $D$ such that
	\[
		\psi_p(D) < \varepsilon,
	\]
	then for all sufficiently large $n$
	\[
		S_n(D) < 2\varepsilon
	\]
	And, since both functions $S_n(x)$ and $\psi_p(x)$ are non-increasing, we will have
	\[
		|\psi_p(x) - S_n(x)| < 3\varepsilon
	\]
	for every $x \geq D$ for all sufficiently large $n$. This implies that functions $S_n(x)$ uniformly converge to $\psi_p(x)$ on the whole $[0; +\infty)$.

	For a non-negative $x$ by $V_n(x)$ denote the volume of the part of $\omega_n \ell_p^n$ defined by $x_1 \geq x$. For $x > \omega_n$ we assume that $V_n(x) = 0$.

	Note that
	\[
		V_n(x) = \frac{1}{2} - \int_{0}^{d} S_n(t) dt
	\]
	And, since $S_n(t)$ uniformly converge to $\psi_p(x)$ on $[0; D]$, we shall have that $V_n(x)$ uniformly converge to
	\[
		\Lambda_p(x) = \frac{1}{2} - \int_{0}^{d} \psi_p(t) dt
	\]
	on $[0; D]$.

	Also
	\begin{multline*}
		\int_{0}^{\infty} \psi_p(x) dx = \int_{0}^{\infty} e^{\frac{1}{p}} e^{-\left(2 \Gamma\left( 1 + \frac{1}{p}\right) e^{\frac{1}{p}} x\right)^p} dx \\ = \frac{1}{2 \Gamma\left( 1 + \frac{1}{p} \right)} \int_{0}^{\infty} e^{-\left(2 \Gamma\left( 1 + \frac{1}{p}\right) e^{\frac{1}{p}} x\right)^p} \left(\frac{1}{2 \Gamma\left( 1 + \frac{1}{p}\right) e^{\frac{1}{p}}} dx\right) \\ = \frac{1}{2 \Gamma\left( 1 + \frac{1}{p} \right)} \int_{0}^{\infty} e^{-x^p} dx = \frac{1}{2 \Gamma\left( 1 + \frac{1}{p} \right)} \int_{0}^{\infty} e^{-t} dt^{\frac{1}{p}} \\ = \frac{1}{2 p \Gamma\left( 1 + \frac{1}{p} \right)} \int_{0}^{\infty} t^{\frac{1}{p} - 1} e^{-t} dt = \frac{\Gamma\left(\frac{1}{p}\right)}{2 p \Gamma\left( 1 + \frac{1}{p} \right)} = \frac{1}{2}
	\end{multline*}

	For an arbitrary $\varepsilon$ there is a number $D$ such that
	\[
		\Lambda_p(D) < \varepsilon
	\]
	Since $V_n(D)$ converge to $\Lambda_p(D)$, we shall have
	\[
		V_n(D) < 2\varepsilon
	\]
	for all sufficiently large $n$. And, since both functions $V_n(x)$ and $\Lambda_p(x)$ are non-increasing,
	\[
		|\Lambda_p(x) - V_n(x)| < 3\varepsilon
	\]
	for every $x \geq D$ for all sufficiently large $n$. This means that functions $V_n(x)$ uniformly converge to
	\[
		\Lambda_p(x) = \int_{x}^{\infty} \phi_p(t) dt
	\]
	on the whole $[0; +\infty)$.

	For Section \ref{euclidean_balls} note that
	\[
		\psi_2(x) = \sqrt{e} e^{-\pi e x^2} = \Psi^{\prime}(-x)
	\]
	\[
		\int_{x}^{\infty} \sqrt{e} e^{-\pi e t^2} dt = \int_{x}^{\infty} e^{-\pi (\sqrt{e} t)^2} d(\sqrt{e} t) = \int_{\sqrt{e} x}^{\infty} e^{-\pi x^2} dx = \Psi(-x)
	\]

	For convenience we introduce new functions defined on $(-\infty; +\infty)$
	\[
		\phi_p(x) = e^{-c_p |x|^p} \quad c_p = 2^p \Gamma\left(1 + \frac{1}{p}\right)^p
	\]
	\[
		\Phi_p(a) = \int_{-\infty}^{a} e^{-c_p |x|^p} dx
	\]
	\[
		\Psi_p(a) = \int_{-\infty}^{a} \phi_p(x) dx
	\]
	The latter function is the reflection of $\Lambda_p$.

	We would like to show that
	\begin{equation}
		\label{phi_p_lim}
		\lim_{\varepsilon \to 0} \frac{c_p^{-\frac{1}{p}}(-\ln \varepsilon)^{\frac{1}{p}}}{-\Phi_p^{-1}(\varepsilon)} = 1
	\end{equation}

	Note that
	\begin{multline*}
		\frac{d}{dx} \frac{e^{-c_p x^p}}{p c_p x^{p - 1}} = \frac{-p c_p x^{p - 1} e^{-c_p x^p}}{p c_p x^{p - 1}} - \frac{e^{-c_p x^p} p(p - 1) c_p x^{p - 2}}{p^2 c_p^2 x^{2(p - 1)}} \\ = - e^{-c_p x^p} \left(1 + \frac{p - 1}{p} \frac{1}{c_p} \frac{1}{x^p} \right)	
	\end{multline*}
	Let
	\[
		A = \frac{p - 1}{p} \frac{1}{c_p}
	\]
	We conclude that for a negative $x$ we shall have
	\[
		\left(\frac{e^{-c_p |x|^p}}{p c_p |x|^{p - 1}}\right)^{\prime} = e^{-c_p |x|^p} \left(1 + \frac{A}{|x|^p} \right)
	\]
	Integrating both sides from $-\infty$ to $a < 0$ gives us
	\begin{equation}
		\label{phi_p_int}
		\frac{e^{-c_p |x|^p}}{p c_p |x|^{p - 1}} = \int_{-\infty}^{a} e^{-c_p |x|^p} \left(1 + \frac{A}{|x|^p} \right) dx
	\end{equation}

	From this we can derive an upper bound on $\Phi_p$
	\[
		\Phi_p(a) = \int_{-\infty}^{a} e^{-c_p |x|^p} dx \leq \frac{e^{-c_p |x|^p}}{p c_p |x|^{p - 1}}
	\]
	Applying $-\Phi_p^{-1}$ to both sides should give us
	\[
		-a \geq -\Phi_p^{-1}\left(\frac{e^{-c_p |x|^p}}{p c_p |x|^{p - 1}}\right)
	\]

	If we put $\varepsilon = \frac{e^{-c_p |x|^p}}{p c_p |x|^{p - 1}}$ into the expression from (\ref{phi_p_lim}), we get
	\[
		\frac{c_p^{-\frac{1}{p}}(c_p |x|^p + \ln(p c_p |x|^{p - 1}))^{\frac{1}{p}}}{-\Phi_p^{-1}\left(\frac{e^{-c_p |x|^p}}{p c_p |x|^{p - 1}}\right)} \geq \frac{c_p^{-\frac{1}{p}}(c_p |x|^p + \ln(p c_p |x|^{p - 1}))^{\frac{1}{p}}}{-a}
	\]
	But as $a$ goes to $-\infty$
	\[
		\frac{e^{-c_p |x|^p}}{p c_p |x|^{p - 1}} \to 0 \textrm{ and } \frac{c_p^{-\frac{1}{p}}(c_p |x|^p + \ln(p c_p |x|^{p - 1}))^{\frac{1}{p}}}{-a} \to 1,
	\]
	from which we conclude
	\begin{equation}
		\label{phi_p_liminf}
		\liminf_{\varepsilon \to 0} \frac{c_p^{-\frac{1}{p}}(-\ln \varepsilon)^{\frac{1}{p}}}{-\Phi_p^{-1}(\varepsilon)} \geq 1
	\end{equation}

	Equality (\ref{phi_p_int}) also leads to a lower bound on $\Phi_p$
	\[
		\Phi_p(a) \geq \left(1 + \frac{A}{|a|^p}\right)^{-1} \frac{e^{-c_p |x|^p}}{p c_p |x|^{p - 1}}
	\]
	By applying $-\Phi_p^{-1}$ to both sides again we get
	\[
		-a \leq -\Phi_p^{-1}\left(\left(1 + \frac{A}{|a|^p}\right)^{-1} \frac{e^{-c_p |x|^p}}{p c_p |x|^{p - 1}}\right)
	\]
	By putting $\varepsilon = \left(1 + \frac{A}{|a|^p}\right)^{-1} \frac{e^{-c_p |x|^p}}{p c_p |x|^{p - 1}}$ into the expression from (\ref{phi_p_lim}) we arrive at
	\begin{multline*}
		\frac{c_p^{-\frac{1}{p}}\left(c_p |x|^p + \ln(p c_p |x|^{p - 1}) + \ln\left(1 + \frac{A}{|a|^p}\right)\right)^{\frac{1}{p}}}{-\Phi_p^{-1}\left(\left(1 + \frac{A}{|a|^p}\right)^{-1} \frac{e^{-c_p |x|^p}}{p c_p |x|^{p - 1}}\right)} \\ \leq \frac{c_p^{-\frac{1}{p}}\left(c_p |x|^p + \ln(p c_p |x|^{p - 1}) + \ln\left(1 + \frac{A}{|a|^p}\right)\right)^{\frac{1}{p}}}{-a}
	\end{multline*}
	As $a$ goes to $-\infty$
	\[
		\left(1 + \frac{A}{|a|^p}\right)^{-1} \frac{e^{-c_p |x|^p}}{p c_p |x|^{p - 1}} \to 0
	\]
	\[
		\frac{c_p^{-\frac{1}{p}}\left(c_p |x|^p + \ln(p c_p |x|^{p - 1}) + \ln\left(1 + \frac{A}{|a|^p}\right)\right)^{\frac{1}{p}}}{-a} \to 1
	\]
	And we conclude
	\[
		\limsup_{\varepsilon \to 0} \frac{c_p^{-\frac{1}{p}}(-\ln \varepsilon)^{\frac{1}{p}}}{-\Phi_p^{-1}(\varepsilon)} \leq 1
	\]

	Together with (\ref{phi_p_liminf}) this leads to
	\[
		\lim_{\varepsilon \to 0} \frac{c_p^{-\frac{1}{p}}(-\ln \varepsilon)^{\frac{1}{p}}}{-\Phi_p^{-1}(\varepsilon)} = 1
	\]

	If by $\Psi_p$ we denote the function
	\[
		\Psi_p(a) = \int_{\infty}^{a} e^{\frac{1}{p}} e^{-\left(2 \Gamma\left( 1 + \frac{1}{p}\right) e^{\frac{1}{p}} x\right)^p} dx = \Phi_p(e^{\frac{1}{p}} a),
	\]
	then we shall have
	\[
		-\Psi^{-1}_p(\varepsilon) \sim \frac{1}{2 e^{\frac{1}{p}} \Gamma\left(1 + \frac{1}{p}\right)} (-\ln \varepsilon)^{\frac{1}{p}}
	\]

	\section{Average distance.}
	\label{av_dist}

	The radius of the unit-volume euclidean $n$-ball is
	\begin{equation}
		\label{omega_asymp}
		\omega_n = \frac{\Gamma\left(\frac{n}{2} + 1\right)^{\frac{1}{n}}}{\sqrt{\pi}} \sim \sqrt{\frac{n}{2\pi e}}
	\end{equation}
	For an arbitrary $\alpha \in (0; 1)$ define
	\begin{equation}
		\label{r_n_def}
		r_n = \alpha \sqrt{\frac{n}{2\pi e}}
	\end{equation}
	By (\ref{omega_asymp}) we have
	\[
		\lim_{n \to \infty} \frac{r_n}{\omega_n} = \alpha < 1
	\]
	If by $V_n$ we denote the volume of the $n$-ball of radius $r_n$, then
	\begin{equation}
		\label{v_n_lim}
		\lim_{n \to \infty} V_n = 0
	\end{equation}

	Now consider an arbitrary unit-volume convex body $K$. For each point $p \in K$ the set of points of $K$ that are at a distance at most $r_n$ from $p$ is of volume not greater than $V_n$, so the average distance $d(K)$ between points in $K$ should be at least
	\[
		d(K) \geq (1 - V_n) r_n
	\]
	And by (\ref{r_n_def}), (\ref{v_n_lim})
	\[
		(1 - V_n) r_n \sim \alpha \sqrt{\frac{n}{2\pi e}}
	\]

	Since the choice of $\alpha \in (0; 1)$ above was arbitrary, for a family of unit-volume convex bodies $K_n$ we should have
	\[
		\liminf_{n \to \infty} \frac{d(K_n)}{\sqrt{\frac{n}{2\pi e}}} \geq 1
	\]
\end{document}